\documentclass[11pt]{article}
\usepackage{latexsym}
\usepackage{graphicx} 
\usepackage{color} 
\usepackage{amsmath, amsthm, amssymb}
\usepackage{enumerate}
\usepackage{float}
\usepackage{url}
\usepackage{stackrel}
\usepackage{mathrsfs,dsfont}
\usepackage{subcaption}
\usepackage{wrapfig}
\usepackage{bbm}
\usepackage{adjustbox}

\usepackage{tikz}
\usetikzlibrary{calc,arrows}

\usepackage[font=scriptsize]{caption}

\usepackage{chemfig,siunitx}
\setcompoundsep{7em} 

\newcommand{\invtPoly}{\mathcal{P}} 

\newcommand{\Rplus}{\mathbb{R}_{>0}}

\newtheorem{theorem}{Theorem}[section]

\newtheorem{proposition}[theorem]{Proposition}
\newtheorem{lemma}[theorem]{Lemma}
\newtheorem{example}[theorem]{Example}

\newtheorem{definition}[theorem]{Definition}

\newtheorem{remark}[theorem]{Remark}

\newcommand{\be}{\begin{equation}}
\newcommand{\ee}{\end{equation}}
\newcommand{\bd}{\begin{displaymath}}
\newcommand{\ed}{\end{displaymath}}
\newcommand{\been}{\begin{enumerate}}
\newcommand{\enen}{\end{enumerate}}
\newcommand{\beit}{\begin{itemize}}
\newcommand{\enit}{\end{itemize}}
\def\ba#1\ea{\begin{align}#1\end{align}}

\def\UUU{\mathfrak U}
\def\FFF{\mathfrak F}

\def\fff{\mathfrak f}
\def\SS{\mathcal S}
\def\CC{\mathcal C}
\def\RR{\mathcal R}

\def\S{\mathbb S}

\def\va{{\vect a}}
\def\vb{{\vect b}}
\def\vu{{\vect u}}
\def\vv{{\vect v}}
\def\vy{{\vect y}}
\def\vx{{\vect x}}
\def\vm{{\vect m}}
\def\vk{{\vect k}}
\def\valpha{{\boldsymbol \alpha}}
\def\v0{{\vect 0}}

\def\Net{(\SS,\CC,\RR)}

\def\TD{{\mathfrak T}_{RNDB}}
\def\TS{{\mathfrak T}_{MCDB}}

\def\Om{\Omega}

\newcommand{\R}{\mathbb{R}}

\newcommand{\Z}{\mathbb{Z}}

\def\ra{\rightarrow}
\def\lra{\leftrightarrow}

\def\rlas{\rightleftarrows}
\def\ds{\displaystyle}

\providecommand{\abs}[1]{\lvert#1\rvert}

\def\vect{\bf }

\newcommand{\specialcell}[2][c]{\begin{tabular}[#1]{@{}c@{}}#2\end{tabular}}

\begin{document}

\title{
A detailed balanced reaction network is sufficient but not necessary for its Markov chain to be detailed balanced}
\author{Badal Joshi}
\date{}

\maketitle

\begin{abstract}
Certain chemical reaction networks (CRNs) when modeled as a deterministic dynamical system taken with mass-action kinetics have the property of reaction network detailed balance (RNDB) which is achieved by imposing network-related constraints on the reaction rate constants. Markov chains (whether arising as models of CRNs or otherwise) have their own notion of detailed balance, imposed by the network structure of the graph of the transition matrix of the Markov chain. When considering Markov chains arising from chemical reaction networks with mass-action kinetics, we will refer to this property as Markov chain detailed balance (MCDB). Finally, we refer to the  stochastic analog of RNDB as Whittle stochastic detailed balance (WSDB). It is known that RNDB and WSDB are equivalent. We prove that WSDB and MCDB are also intimately related but are not equivalent. While RNDB implies MCDB, the converse is not true. The conditions on rate constants that result in networks with MCDB but without RNDB are stringent, and thus examples of this phenomenon are rare, a notable exception is a network whose Markov chain is a birth and death process. We give a new algorithm to find conditions on the rate constants that are required for MCDB. 

~
\end{abstract}

\section{Introduction} \label{sec:introduction}

\begin{figure}
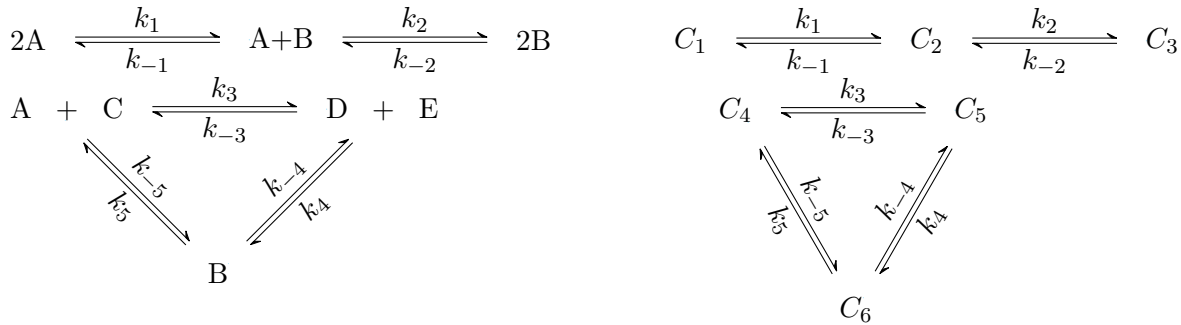

\begin{subfigure}[b]{0.30\textwidth}
\centering
\schemestart
 2A 
 \arrow{<=>[$k_{1}$][$k_{-1}$]}
 A+B 
  \arrow{<=>[$k_{2}$][$k_{-2}$]}
 2B 
 \schemestop
 
 \schemestart
 A \+ C
 \arrow{<=>[$k_{3}$][$k_{-3}$]}
  D \+ E 
 \arrow{<=>[$k_{-4}$][$k_{4}$]}[225]  
 B 
  \arrow{<=>[$k_{-5}$][$k_{5}$]}[135] 
\schemestop
\end{subfigure}
\hspace{1.5in}
\begin{subfigure}[b]{0.30\textwidth}
\centering
\schemestart
$C_1$
 \arrow{<=>[$k_{1}$][$k_{-1}$]}
$C_2$
  \arrow{<=>[$k_{2}$][$k_{-2}$]}
$C_3$
 \schemestop
 
\schemestart
$C_4 $
 \arrow{<=>[$k_{3}$][$k_{-3}$]}
$ C_5$
 \arrow{<=>[$k_{-4}$][$k_{4}$]}[240]  
$C_6$
  \arrow{<=>[$k_{-5}$][$k_{5}$]}[120] 
\schemestop
\end{subfigure}
\caption{Example of a reversible chemical reaction network with 5 chemical species ($\{A,B,C,D,E\}$), 6 complexes ($\{2A, A+B, 2B, A+C, D+E, B \}$) and 5 reversible reaction pairs (edges of the network). On the right, we depict the same reaction network as on the left, but we suppress the species composition of the complexes. Some of the conditions for detailed balance only depend on the network structure shown on the right, while other conditions require the knowledge of the constituents of the complexes $C_i$.} 
\label{fig:crn_ex}
\end{figure}

The concept of detailed balance (see \cite{aris1965prolegomena, aris1968prolegomena, feinberg1989necessary, horn1972general}) arose early in the history of chemical reaction network theory (see for instance \cite{FeinLectures, gunawardena2003chemical, horn1972general}). As an example, consider the reversible chemical reaction network (CRN) in figure \ref{fig:crn_ex}. A CRN is prescribed a kinetic scheme (such as mass-action kinetics) which describes through a system of ODEs the dynamics of variables, in this case the time-dependent concentrations of the chemical species. For a reversible reaction network (a reaction network is reversible if every reaction is reversible), the kinetic scheme assigns a unique function to each reaction pair. For instance, mass-action kinetics assigns the binomial $k_1c_A^2 - k_{-1} c_A c_B$ to the reaction pair $2A \stackrel[k_{-1}]{k_1}{\rlas} A+B$, where $c_A$ and $c_B$ represent the concentrations of the species $A$ and $B$ respectively. An equilibrium of the dynamical system is said to be detailed balanced, if at the steady state concentration, the function corresponding to each reaction pair vanishes. Thus, if $c_A^*$ and $c_B^*$ represent the equilibrium concentrations of the species $A$ and $B$ respectively, then one of the conditions required for the mass-action equilibrium to be detailed balanced is that $k_1c_A^{*2} = k_{-1} c_A^* c_B^*$. The full set of conditions arising from the five reaction pairs in the network in figure \ref{fig:crn_ex} is as follows:
\begin{align} \label{eq:crn_ex}
k_1c_A^{*2} &= k_{-1} c_A^* c_B^* \nonumber \\
k_2 c_A^*c_B^* &= k_{-2} c_B^{*2} \nonumber \\
k_3 c_A^* c_C^* &= k_{-3} c_D^* c_E^* \nonumber \\
k_4 c_D^* c_E^* &= k_{-4} c_B^* \nonumber \\
k_5 c_B^* &= k_{-5} c_A^* c_C^*
\end{align}
Assuming that a detailed balanced equilibrium exists for this network, the set of conditions \eqref{eq:crn_ex} can be simultaneously satisfied only if the reaction rate constants are appropriately constrained. Some algebra then leads us to identify the constraints as follows
\begin{align} 
\frac{k_1}{k_{-1}} &= \frac{k_2}{k_{-2}} \label{eq:crn_ex_cons1} \\
k_3 k_4 k_5 &= k_{-3} k_{-4} k_{-5} \label{eq:crn_ex_cons2}
\end{align}
Note that the second constraint \eqref{eq:crn_ex_cons2} does not require the precise knowledge of the species constituents of the complexes, but only the fact that the triplet of complexes $\{C_4 := A+C, C_5 := D+E, C_6:=B\}$ forms a cycle as is evident in the depiction on the right of figure \ref{fig:crn_ex}. On the other hand, \eqref{eq:crn_ex_cons1} cannot be obtained without knowing that $C_1=2A$, $C_2=A+B$ and $C_3=2B$. The set of constraints on the rate constants, such as the one in \eqref{eq:crn_ex_cons1} and \eqref{eq:crn_ex_cons2}, is not only necessary but also sufficient in the sense that if the constraints are satisfied for a reversible CRN with mass-action kinetics, then every equilibrium of the CRN is detailed balanced. Thus we say that the CRN is detailed balanced when its equilibria are detailed balanced or equivalently when the rate constants satisfy the appropriate constraints, called {\em circuit conditions} in \cite{feinberg1989necessary}. If a CRN has this form of detailed balance, then we will say that the CRN satisfies Reaction Network Detailed Balance (RNDB).

\begin{figure}[h!]
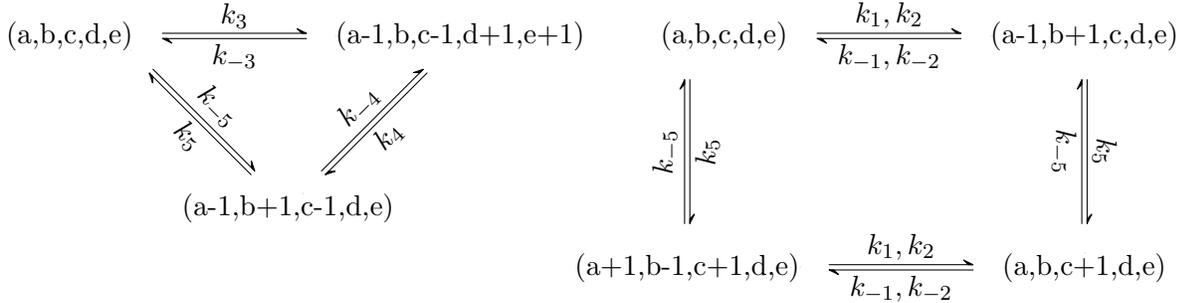

\begin{subfigure}[b]{0.30\textwidth}
\centering

 \schemestart
 (a,b,c,d,e)
 \arrow(1--2){<=>[$k_{3}$][$k_{-3}$]}
 (a-1,b,c-1,d+1,e+1)
 \arrow(2--3){<=>[$k_{-4}$][$k_{4}$]}[-135]  
 (a-1,b+1,c-1,d,e)
\arrow(3--4){<=>[$k_{-5}$][$k_{5}$]}[135]  
\schemestop
\end{subfigure}
\hspace{1in}
\begin{subfigure}[b]{0.30\textwidth}
\centering
 \schemestart
 (a,b,c,d,e)
 \arrow(1--2){<=>[$k_{1}, k_2$][$k_{-1}, k_{-2}$]}
 (a-1,b+1,c,d,e)
 \arrow(2--3){<=>[$k_5$][$k_{-5}$]}[-90]  
 (a,b,c+1,d,e)
\arrow(3--4){<=>[$k_{1},k_{2}$][$k_{-1},k_{-2}$]}[180] 
 (a+1,b-1,c+1,d,e)
\arrow(4--1){<=>[$k_{-5}$][$k_{5}$]}[90]  
\schemestop
\end{subfigure}
\caption{Two possible cycle types that arises in the graph of the Markov chain of the chemical reaction network depicted in figure \ref{fig:crn_ex}. The cycle type on the left owes its topology to the reaction cycle $\{A+C \lra D+E \lra B \lra A+C \}$. However, the cycle type on the right does not arise from a reaction cycle.} \label{fig:crn_ex_mc}
\end{figure}

An unrelated notion of detailed balance arose in the literature on continuous-time Markov chains, also in its early history \cite{durrett2010probability,kelly1979reversibility}. The equilibrium of a Markov chain (MC), if it exists, is a  stationary measure. A Markov chain is defined via the transition rates $\rho(x,y)$ between states $x$ and $y$. A stationary measure $\mu$ (which is a non-negative, countably additive function that satisfies $\sum_{x} \mu(x) \rho(x,y) = \mu(y)$) is detailed balanced if every transition is reversible ($\rho(y,x) > 0$ if $\rho(x,y) > 0$ for all states $x$ and $y$) and $\mu(x) \rho(x,y) = \mu(y) \rho(y,x)$ for all states $x$ and $y$. A Markov chain has a detailed balanced stationary measure $\mu$ if and only if every transition is reversible and the Markov chain satisfies the {\em Kolmogorov cycle condition}, which states that the product of the transition rates over every cycle in the graph of the Markov chain is independent of the direction in which the cycle is traversed. Thus a Markov chain which satisfies the Kolmogorov cycle condition over every cycle can be considered a detailed balanced Markov chain.

Despite the fact that the two notions of detailed balance discussed in the previous paragraphs are unrelated, and the two arise independently in distinct contexts, there is a natural reason that they possess the same name. A CRN has a natural graph structure associated with it (see figure \ref{fig:crn_ex}) and so does a Markov chain. In the case of CRNs, the nodes are chemical complexes (which can be thought of as multisets of species, for instance $2A, A+B, 2B, A+C, D+E, B$ are complexes in \ref{fig:crn_ex}). The directed edges are the reactions, where mass-action kinetics provides a positive valued labeling to the edges called the reaction rate constant. For instance, mass-action kinetics associates the rate constant $k_2$ and the reaction rate $k_2 c_A c_B$ with the directed edge $A+B \stackrel[]{k_2}{\to} 2B$ in figure \ref{fig:crn_ex}. Detailed balance here refers to the property that at equilibrium (which is the fixed point of the dynamical system), the reaction rate for each reaction is balanced by the reaction rate for the corresponding reverse reaction. A Markov chain (in any setting, not necessarily that of CRN) has a natural graph structure associated with it as well. The nodes are the states and the directed edges are the transitions that occur at a positive rate. If at equilibrium (which is the stationary measure of the Markov chain), the flow $\mu(a) \rho(a,b)$ from state $a$ to state $b$ is balanced by the reverse flow $\mu(b) \rho(b,a)$ for every pair of states $\{a,b\}$ then we say the Markov chain is detailed balanced. Thus it is appropriate to consider the two kinds of detailed balance discussed here are homonyms, which acquire their precise definition from the context of the graphical structure on which they are defined.  

When we consider a stochastic model of a chemical reaction network (which is a continuous-time Markov chain), both graphical structures co-exist. Specifically, the graph structure of the CRN itself, such as the one depicted in figure \ref{fig:crn_ex}, is present in the background. The graph structure of the Markov chain is defined as follows: the nodes are elements of $\Z_{\ge 0}^s$ where the $i$-th component of each node represents the number of molecules of the $i$-th species. A directed edge $(a \to b)$ represents a positive rate of transition between nodes $a$ and $b$, which in turn exists if there is a reaction in the CRN that allows this transition. Thus the graph structure in the foreground of the stochastic model is the one associated with the Markov chain and it inherits some structural elements from the graph of the CRN. When considering a stochastic model of a CRN with mass-action kinetics, if the corresponding Markov chain has the property of detailed balance, we say that the CRN satisfies Markov chain detailed balance (MCDB). 

There are, however, some critical distinctions between the two graphs in the stochastic model of a CRN. Firstly, the nodes in the CRN graph are chemical complexes, while the nodes in the corresponding Markov chain represent the number of molecules of the species. Secondly, the graph of a CRN is finite, while the graph of the corresponding Markov chain usually has an infinite state space. Thirdly, some cycles in the Markov chain are inherited from the cycles in the CRN, but certain cycles arise in the Markov chain from a set of reactions that do not themselves form a cycle in the CRN. For instance, the cycle on the left in figure \ref{fig:crn_ex_mc} is clearly inherited from the reaction cycle $\{A+C \lra D+E \lra B \lra A+C \}$ in figure \ref{fig:crn_ex}, but there is no equivalent reaction cycle in figure \ref{fig:crn_ex} of the cycle on the right in figure \ref{fig:crn_ex_mc}.  

The question then emerges: does detailed balance in a Markov chain model of a CRN with mass-action kinetics have any relation with detailed balance in the ODE model of a CRN with mass-action kinetics? Having considered the distinctions mentioned in the previous paragraph, it seems rather surprising that RNDB and MCDB are even related. However, it appears to be a ``folk theorem" that the conditions for RNDB and MCDB are equivalent. The primary purpose of this article is to lay to rest  this false notion and establish that RNDB and MCDB are not equivalent. A possible reason for the misunderstanding arising in the first place, may be that a concept of stochastic detailed balance specific to CRNs (which is not the usual Markov chain detailed balance) was introduced by Whittle \cite{whittle1986systems} (see Definition \ref{def:wsdb}), as an analogue of RNDB. In order to disambiguate, we refer to this notion as Whittle Stochastic Detailed Balance (WSDB). Unfortunately, WSDB and MCDB are often confused with each other (see Theorem 4.5 in \cite{anderson2010product}). We emphasize that WSDB is not equivalent to MCDB for general CRNs, furthermore, the content of the latter is sensible for all Markov chains, while the former is specific to CRNs. The main result of this article may be summarized as follows:  \\
{\em For reversible reaction networks, 
\begin{align*}
\boxed{WSDB \iff RNDB \implies MCDB}
\end{align*}}

We show that the conditions on the rate constants required for RNDB imply MCDB, however the converse is not true in general. Strikingly however, in a vast majority of CRNs, the conditions for RNDB do coincide with those for MCDB, which is perhaps another reason that the two types of detailed balance are often confused. Despite this, it does not take much effort to construct an example of a CRN where RNDB and MCDB do not require identical conditions on the rate constants. Consider, for instance, the CRN $2A \lra A+B \lra 2B$, which satisfies MCDB independent of rate constants, however does not satisfy RNDB for almost any set of randomly chosen rate constants. We study this CRN and two of its variants in section \ref{sec:motexamples}. 

This article is organized as follows:  section \ref{sec:motexamples} presents certain examples of networks that motivate the main results of the paper; section \ref{sec:intro_CRN} provides an overview of chemical reaction network theory and mass-action kinetics in the context of deterministic and stochastic models of chemical reaction networks; section \ref{sec:detbal} describes detailed balance as a collection of ideas and then in specific contexts of deterministic and stochastic models of CRN; section \ref{sec:mainthm} elucidates the relation between reaction network detailed balance and Markov chain detailed balance, and states the main theorem that RNDB implies MCDB; section \ref{sec:applications} presents a new algorithm to find constraints on the rate constants that result in MCDB and gives an explicit formulation of the stationary distribution of the stochastic model when the network has reaction network detailed balance. 

\section{Motivational examples}\label{sec:motexamples}

We present some examples that will motivate the main results in this article. The first three networks are very similar to each other, however, the relation between RNDB and MCDB is quite different for the three networks. 


\begin{table}[h!]
\centering
 \begin{adjustbox}{max width=1\textwidth}
\begin{tabular}{|c || c |c | c |c|}
\hline
~ & Network & RNDB holds when & MCDB holds when & Relation  \\
\hline
\hline

1 & \specialcell{$0 \stackrel[k_{-1}]{k_1}{\rlas}  A , ~~ B \stackrel[k_{-4}]{k_4}{\rlas} 0$ \\ $2A \stackrel[k_{-2}]{k_2}{\rlas} A+B  \stackrel[k_{-3}]{k_3}{\rlas} 2B  $} 
& 
\specialcell{$\ds \frac{k_2}{k_{-2}} = \frac{k_3}{k_{-3}}$ \\   
$k_1 k_3 k_4 = k_{-1} k_{-3} k_{-4}$}   
& 
\specialcell{$\ds \frac{k_2}{k_{-2}} = \frac{k_3}{k_{-3}}$ \\   
$k_1 k_3 k_4 = k_{-1} k_{-3} k_{-4}$}   
&
RNDB $\iff$ MCDB \\
\hline

2 & \specialcell{$0 \stackrel[k_{-1}]{k_1}{\rlas}  A$  \\ $2A \stackrel[k_{-2}]{k_2}{\rlas} A+B  \stackrel[k_{-3}]{k_3}{\rlas} 2B $} 
& 
$\ds \frac{k_2}{k_{-2}} = \frac{k_3}{k_{-3}}$   
& 
$\ds \frac{k_2}{k_{-2}} = \frac{k_3}{k_{-3}}$   
&
RNDB $\iff$ MCDB \\

\hline
3 & \specialcell{$2A \stackrel[k_{-2}]{k_2}{\rlas} A+B  \stackrel[k_{-3}]{k_3}{\rlas} 2B $} 
& 
$\frac{k_2}{k_{-2}} = \frac{k_3}{k_{-3}}$   
& 
No conditions   
&
RNDB $\implies$ MCDB \\

\hline 
4 & \specialcell{$0 \stackrel[k_{-1}]{k_1}{\rlas} A, 2A \stackrel[k_{-2}]{k_2}{\rlas} 3A$ \\ 
$A \stackrel[k_{-3}]{k_3}{\rlas} A+B, 2B \stackrel[k_{-4}]{k_4}{\rlas} 3B$}
& 
\specialcell{$\frac{k_1}{k_{-1}} = \frac{k_2}{k_{-2}}$ \\ $\frac{k_3}{k_{-3}} = \frac{k_4}{k_{-4}}$} 
& 
$\frac{k_3}{k_{-3}} = \frac{k_4}{k_{-4}}$
&
RNDB $\implies$ MCDB \\

\hline
5 & \specialcell{$2A \stackrel[k_{-1}]{k_1}{\rlas} A+B  \stackrel[k_{-2}]{k_2}{\rlas} 2B$ \\ 
$A+C \stackrel[k_{-3}]{k_3}{\rlas} D+E  \stackrel[k_{-4}]{k_4}{\rlas}  B \stackrel[k_{-5}]{k_5}{\rlas}  A+C$} 
& 
\specialcell{$\frac{k_1}{k_{-1}} = \frac{k_2}{k_{-2}}$  \\
$k_3 k_4 k_5 = k_{-3} k_{-4} k_{-5}$ } 
& 
\specialcell{$\frac{k_1}{k_{-1}} = \frac{k_2}{k_{-2}}$  \\
$k_3 k_4 k_5 = k_{-3} k_{-4} k_{-5}$ } 
&
RNDB $\iff$ MCDB \\

\hline
6 & \specialcell{$0 \stackrel[k_{-1}]{k_1}{\rlas} A, 2A  \stackrel[k_{-2}]{k_2}{\rlas} 3A$ \\ 
$0 \stackrel[k_{-3}]{k_3}{\rlas} B, 2B \stackrel[k_{-4}]{k_4}{\rlas} 3B$} 
& 
\specialcell{$\frac{k_1}{k_{-1}} = \frac{k_2}{k_{-2}}$ \\ $\frac{k_3}{k_{-3}} = \frac{k_4}{k_{-4}}$} 
& 
No conditions
&
RNDB $\implies$ MCDB \\
\hline
\end{tabular}
\end{adjustbox}
\caption{For each example network: (1) we list the constraints on the rate constants of the chemical reaction network which results in detailed balance in deterministic system endowed with mass-action kinetics (abbreviated as RNDB), and (2) we list the constraints on the rate constants of the chemical reaction network which result in detailed balance in the Markov chain model when endowed with stochastic mass-action kinetics (abbreviated as MCDB). The examples always have RNDB $\implies$ MCDB and sometimes have RNDB $\iff$ MCDB. }
\label{tab:examples}
\end{table}

\subsection{Network 1}

Consider the following chemical reaction network
\begin{align} \label{eq:net1}
0 \stackrel[k_{-1}]{k_1}{\rlas}  A \quad, \quad 2A \stackrel[k_{-2}]{k_2}{\rlas} A+B  \stackrel[k_{-3}]{k_3}{\rlas} 2B   \quad, \quad  B \stackrel[k_{-4}]{k_4}{\rlas} 0
\end{align}

We wish to calculate conditions on the reaction rate constants which will ensure reaction network detailed balance (RNDB) (see section \ref{sec:RNDB} for definition) under the assumption of mass-action kinetics (see section \ref{sec:mak} for definition). RNDB is ensured if each reaction is reversible and if at steady state concentration the reaction rate of each forward reaction is equal to the reaction rate of the corresponding backward reaction. Applying this condition to \eqref{eq:net1}, we get the following set of constraints on the reaction rate constants:
\begin{align} \label{eq:net1cond1}
k_1 = k_{-1} c_A^*, ~ k_2 c_A^* = k_{-2} c_B^*, ~ k_3 c_A^* = k_{-3} c_B^*, ~ k_4 c_B^* = k_{-4}
\end{align}
where $c_A^*$ and $c_B^*$ represent the steady state concentrations of the chemical species $A$ and $B$, in other words concentrations which satisfy $\ds \frac{d}{dt}c_A^* = \frac{d}{dt}c_B^* = 0$. The conditions in \eqref{eq:net1cond1} can be satisfied simultaneously if and only if the following relations hold among the reaction rate constants:
\begin{align} \label{eq:net1cond2}
k_1 k_2 k_4 = k_{-1} k_{-2} k_{-4}, ~ k_1 k_3 k_4 = k_{-1} k_{-3} k_{-4}  
\end{align}
An alternate way to write these relations is as follows
\begin{align} \label{eq:net1cond3}
\frac{k_2}{k_{-2}} = \frac{k_3}{k_{-3}}, ~ 
k_1 k_3 k_4 = k_{-1} k_{-3} k_{-4}   
\end{align}
where the first relation is a spanning forest condition and the second relation is a circuit condition in the nomenclature of \cite{feinberg1989necessary}. 

\begin{figure}[h!]
\centering
\begin{tikzpicture}[scale=4]
\draw [<->, line width=2, red] (1,0) -- (0,1);
\draw [<->, line width=2, green] (0,0) -- (1,0);
\draw [<->, line width=2, blue] (0,1) -- (0,0);

\node [below left] at (0,0) {$(a,b)$};
\node [below right] at (1,0) {$(a+1,b)$};
\node [above left] at (0,1) {$(a,b+1)$};

\node [below] at (0.5,0) {$k_1 {\color{green}  ~\circlearrowleft} $};
\node [below] at (0.5,-0.1) {$k_{-1}(a+1) {\color{green} ~\circlearrowright} $};

\node [above right] at (0.4,0.6) {$k_2(a+1)a + k_3(a+1)b {\color{red} ~ \circlearrowleft} $};
\node [above right] at (0.6,0.4) {$k_{-2}a(b+1) + k_{-3}(b+1)b {\color{red} ~ \circlearrowright} $};

\node [left] at (0,0.6) {$k_4(b+1) {\color{blue} ~\circlearrowleft} $};
\node [left] at (0,0.4) {$k_{-4} {\color{blue} ~\circlearrowright} $};
\end{tikzpicture}
\caption{Rates of transition between nodes for the network $\left\{0 \stackrel[k_{-1}]{k_1}{\rlas}  A , 2A \stackrel[k_{-2}]{k_2}{\rlas} A+B  \stackrel[k_{-3}]{k_3}{\rlas} 2B   ,  B \stackrel[k_{-4}]{k_4}{\rlas} 0\right\}$. The diagram depicts one of the types of reaction cycles occurring in the graph of the Markov chain. $(a,b) \in \Z_{\ge 0}$ denotes the population numbers of the species $A$ and $B$. $\circlearrowleft$ indicates the rate of transition while going counterclockwise and $\circlearrowright$ indicates the rate of transition while going clockwise. For instance $k_1$ is the transition rate from $(a,b)$ to $(a+1,b)$ and $k_{-1}(a+1)$ is the transition rate from $(a+1,b)$ to $(a,b)$. To apply the Kolmogorov cycle condition, we simply set the product of the reaction rates on the path going clockwise $\circlearrowright$ equal to the product of the reaction rates on the path going counterclockwise $\circlearrowleft$. For this example, this results in $[k_1] [k_2(a+1)a + k_3(a+1)b] [k_4(b+1)] = [k_{-1}(a+1)] [k_{-2}a(b+1) + k_{-3}(b+1)b] [k_{-4}]$, which simplifies to $\ds \frac{k_1}{k_{-1}} \cdot \frac{k_2a + k_3b}{k_{-2}a + k_{-3}b} \cdot \frac{k_4}{k_{-4}} = 1$, a relation that holds for all $a \ge 1$ and for all $b \ge 1$.} \label{fig:kcc}
\end{figure}

We now wish to compare the results of the previous calculation with the  conditions on the reaction rate constants which ensure Markov chain detailed balance (MCDB) (see section \ref{sec:MCDB} for definition). Given a chemical reaction network with mass-action kinetics, within a stochastic framework, the state of the system at time $t$ is an element of $\Z_{\ge 0}^n$ representing the number of molecules of each of the $n$ chemical species at time $t$. Each instance of a chemical reaction changes the population numbers along the reaction vector corresponding to that reaction. This process is naturally modeled as a Markov chain which has a graph structure associated with it, where each directed edge corresponds to a transition that occurs at a non-zero rate. MCDB is ensured if the Kolmogorov cycle condition (KCC) holds. KCC is said to hold if each cycle in the graph of the Markov chain can be traversed in either direction, and for each cycle $C$ the product of the transition rates when traversing $C$ clockwise is equal to the product of transition rates when traversing $C$ counterclockwise. For the network in \eqref{eq:net1}, let $(a,b)$, which represents the number of molecules of species $A$ and species $B$ respectively, be the state of the system. For the cycle whose node set is $\{(a,b), (a+1,b), (a,b+1) \}$, KCC is equivalent to the following condition for all $a \ge 1$ and $b \ge 1$
\begin{align} \label{eq:net1kcc1}
\frac{k_1}{k_{-1}} \cdot \frac{k_2a + k_3b}{k_{-2}a + k_{-3}b} \cdot \frac{k_4}{k_{-4}} = 1
\end{align}
These conditions hold for all $a \ge 1$ and all $b \ge 1$ if and only if the following relations apply to the reaction rate constants:
\begin{align} \label{eq:net1kcc2}
k_1 k_2 k_4 = k_{-1} k_{-2} k_{-4}, ~ k_1 k_3 k_4 = k_{-1} k_{-3} k_{-4}  
\end{align}
a set of relations that is identical to \eqref{eq:net1cond2}. There are other cycle types in the graph of the Markov chain which we have not considered, for instance the cycle that is obtained by considering the transitions resulting from the reactions $\{0 \to A, 0 \to B, A \to 0, B \to 0 \}$ (see also figure \ref{fig:cycletypes}). However, it is straightforward to show that this cycle, or any other cycle, does not result in any new relations among the rate constants. Thus for the network in \eqref{eq:net1} taken with mass-action kinetics, the minimal set of relations among the rate constants that guarantees RNDB is identical to the set of relations that guarantees MCDB. In other words, one may say that RNDB is equivalent to MCDB for the network in \eqref{eq:net1}. 

\begin{figure}[h!]
\centering
\begin{tikzpicture}[scale=1.75]
\draw[help lines, dashed, line width=1] (0,0) grid (3,3);
\draw [<->, line width=1.5, black] (0,3.5) -- (0,0) -- (3.5,0);
\draw [<->, line width=1.25, red] (1,1) -- (0,2);
\draw [<->, line width=1.25, red] (2,1) -- (1,2);
\draw [<->, line width=1.25, red] (3,1) -- (2,2);
\draw [<->, line width=1.25, red] (1,2) -- (0,3);
\draw [<->, line width=1.25, red] (2,2) -- (1,3);
\draw [<->, line width=1.25, red] (3,2) -- (2,3);
\draw [<->, line width=1.25, red] (2,0) -- (1,1);
\draw [<->, line width=1.25, red] (3,0) -- (2,1);

\draw [<->, line width=1.25, green] (0,0) -- (1,0);
\draw [<->, line width=1.25, green] (1,0) -- (2,0);
\draw [<->, line width=1.25, green] (2,0) -- (3,0);
\draw [<->, line width=1.25, green] (0,1) -- (1,1);
\draw [<->, line width=1.25, green] (1,1) -- (2,1);
\draw [<->, line width=1.25, green] (2,1) -- (3,1);
\draw [<->, line width=1.25, green] (0,2) -- (1,2);
\draw [<->, line width=1.25, green] (1,2) -- (2,2);
\draw [<->, line width=1.25, green] (2,2) -- (3,2);
\draw [<->, line width=1.25, green] (0,3) -- (1,3);
\draw [<->, line width=1.25, green] (1,3) -- (2,3);
\draw [<->, line width=1.25, green] (2,3) -- (3,3);

\draw [<->, line width=1.25, blue] (0,1) -- (0,0);
\draw [<->, line width=1.25, blue] (0,2) -- (0,1);
\draw [<->, line width=1.25, blue] (0,3) -- (0,2);
\draw [<->, line width=1.25, blue] (1,1) -- (1,0);
\draw [<->, line width=1.25, blue] (1,2) -- (1,1);
\draw [<->, line width=1.25, blue] (1,3) -- (1,2);
\draw [<->, line width=1.25, blue] (2,1) -- (2,0);
\draw [<->, line width=1.25, blue] (2,2) -- (2,1);
\draw [<->, line width=1.25, blue] (2,3) -- (2,2);
\draw [<->, line width=1.25, blue] (3,1) -- (3,0);
\draw [<->, line width=1.25, blue] (3,2) -- (3,1);
\draw [<->, line width=1.25, blue] (3,3) -- (3,2);

\draw [cyan, ultra thick, dashed] (0.5,0.5) circle [radius=0.8];
\node [cyan] at (0.5,0.5) {Type 2};
\node [cyan] at (0.35,2.25) {Type 1};
\draw [cyan, ultra thick, dashed] (0.4,2.4) circle [radius=0.65];
\end{tikzpicture}
\caption{The graph of the Markov chain of a reaction network has multiple cycles.  For the network $\{0 \lra A, 0 \lra B, 2A \lra A+B \lra 2B\}$ there are infinitely many cycles. However, there are only two types of non-trivial cycles in the network. Type 1 cycle arises from the sequence of reactions $\{0 \to A, A+B \to 2B, B \to 0\}$ and its reverse sequence. Type 2 cycle arises from the sequence of reactions $\{0 \to A, 0 \to B, A \to 0, B \to 0 \}$ and its reverse sequence. Other than the cycle $(0,0) \lra (1,0) \lra (1,1) \lra (0,1)$, every other cycle in the graph of this network can be decomposed into Type 1 cycles.}
\label{fig:cycletypes}
\end{figure}
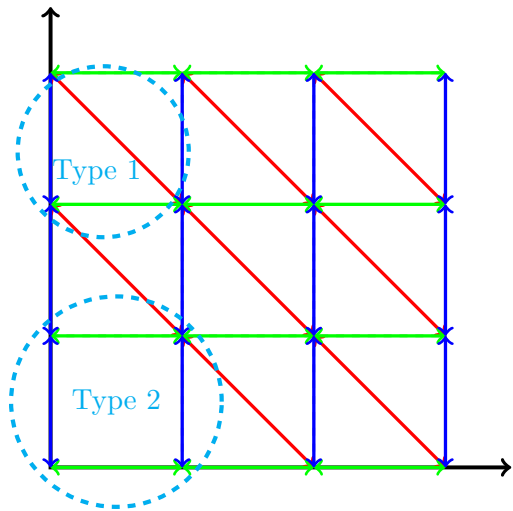

\subsection{Network 2}
Now consider the following chemical reaction network

\begin{align} \label{eq:net2}
2A \stackrel[k_{-2}]{k_2}{\rlas} A+B \stackrel[k_{-3}]{k_3}{\rlas} 2B
\end{align}
This network is similar to the network in \eqref{eq:net1}, except that the flow reactions $\{ 0 \lra A, ~ 0 \lra B \}$ are excluded. We calculate the conditions on the rate constants that ensure RNDB and the conditions on the rate constants that ensure MCDB. Using the same notation as in \eqref{eq:net1}, we find that for RNDB we need:
\begin{align} \label{eq:net2cond1}
k_2 c_A^* = k_{-2} c_B^*, ~ k_3 c_A^* = k_{-3} c_B^*
\end{align}
which implies for the rate constants
\begin{align} \label{eq:net2cond2}
\frac{k_2}{k_{-2}} = \frac{k_3}{k_{-3}} 
\end{align}
Figure \ref{fig:net2mc} depicts the graph of the Markov chain for the network in \eqref{eq:net2}. We note that there are no cycles in the graph, which means that the Kolmogorov cycle condition is vacuously satisfied which implies that MCDB holds irrespective of the values of the reaction rate constants, or indeed, MCDB holds independent of the type of kinetics. 

\begin{figure}[h!]
\centering
\begin{tikzpicture}[scale=1.25] 
\draw[help lines, dashed, line width=1] (0,0) grid (4,4);
\draw [<->, line width=1.5, black] (0,4.5) -- (0,0) -- (4.5,0);
\draw [<->, line width=2, yellow, dotted] (1,1) -- (0,2);
\node [below left, yellow] at (0.5,1.5) {$A+B \lra 2B$};
\draw [<->, line width=2, blue, dashed] (1,1) -- (2,0);
\node [below left, blue] at (1.5,0.5) {$2A \lra A+B$};
\node [above right, green] at (1.5,1.5) {$2A \lra A+B \lra 2B$};
\node [below left] at (0,0) {$(0,0)$};
\node [below] at (1,0) {$(1,0)$};
\node [below] at (2,0) {$(2,0)$};
\node [below] at (3,0) {$(3,0)$};
\node [below] at (4,0) {$(4,0)$};
\node [left] at (0,1) {$(0,1)$};
\node [left] at (0,2) {$(0,2)$};
\node [left] at (0,3) {$(0,3)$};
\node [left] at (0,4) {$(0,4)$};

\draw [<->, line width=2, yellow, dotted] (1,2) -- (0,3);
\draw [<->, line width=2, yellow, dotted] (1,3) -- (0,4);
\draw [<->, line width=2, blue, dashed] (2,1) -- (3,0);
\draw [<->, line width=2, blue, dashed] (3,1) -- (4,0);

\draw [<->, line width=2, green] (1,2) -- (2,1);
\draw [<->, line width=2, green] (1,3) -- (2,2);
\draw [<->, line width=2, green] (3,2) -- (2,3);
\draw [<->, line width=2, green] (3,1) -- (2,2);
\draw [<->, line width=2, green] (4,1) -- (3,2);
\draw [<->, line width=2, green] (4,2) -- (3,3);
\draw [<->, line width=2, green] (4,3) -- (3,4);
\draw [<->, line width=2, green] (2,3) -- (1,4);
\draw [<->, line width=2, green] (3,3) -- (2,4);
\end{tikzpicture}
\caption{The graph of the Markov chain for the network $\{2A \lra A+B \lra 2B \}$. The edges represent transitions that occur with positive probability. The yellow/dotted edges represent the transitions corresponding to the reversible reaction pair $A+B \lra 2B$, the blue/dashed edges represent the transitions corresponding to the reversible reaction pair $2A \lra A+B$, while the green/solid edges represent the transitions corresponding to both the reversible reaction pairs $2A \lra A+B \lra 2B$. Note that given an initial population $(a_0, b_0)$, the population $(a(t), b(t))$ at any future time $t$ is restricted to the line $a(t) + b(t) = a_0 + b_0$, thus the dynamics is restricted to one dimension. Since there are no cycles in the graph of the Markov chain, KCC holds vacuously implying that MCDB holds for the network irrespective of the values of the reaction rate constants.}
 \label{fig:net2mc}
\end{figure}
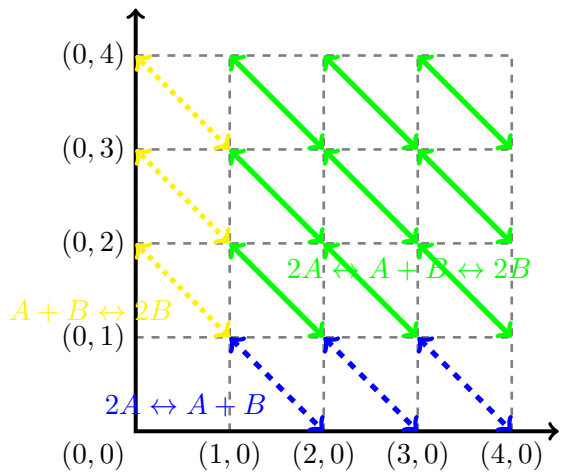

\subsection{Network 3}

On the other hand, consider the following network, which is ``halfway'' between the networks considered in \eqref{eq:net1} and \eqref{eq:net2}:
\begin{align} \label{eq:net1.5}
0 \stackrel[k_{-1}]{k_1}{\rlas}  A \quad, \quad 2A \stackrel[k_{-2}]{k_2}{\rlas} A+B  \stackrel[k_{-3}]{k_3}{\rlas} 2B   
\end{align}
While the flow reactions for species $A$ are present, the flow reactions for species $B$ are absent. For this network, the RNDB conditions are still the same as for the network \eqref{eq:net2}, $\frac{k_2}{k_{-2}} = \frac{k_3}{k_{-3}}$ and it turns out that the MCDB conditions are identical to the RNDB conditions. Thus including the flow reactions $0 \lra A$ does not result in any new conditions on the rate constants for RNDB. However, including the flow reactions $0 \lra A$ is sufficient to introduce cycles into the graph of the Markov chains, which introduces a new condition for MCDB, a condition that was already required for RNDB in the network \eqref{eq:net2} without flow reactions. Thus the introduction of the flow reactions for species $A$ is sufficient to make the conditions for RNDB and MCDB equivalent. 

This example might lead one to conjecture that the presence of cycles in the graph of a Markov chain is sufficient to guarantee that MCDB and RNDB are equivalent. However, the next example provides a counter to this claim since the graph of the Markov chain has cycles and the MCDB conditions are a nonempty, proper subset of the conditions for RNDB. 

\subsection{Network 4}
\begin{align}
0 \stackrel[k_{-1}]{k_1}{\rlas} A \quad, \quad 2A \stackrel[k_{-2}]{k_2}{\rlas} 3A  \quad, \quad
A \stackrel[k_{-3}]{k_3}{\rlas} A+B \quad, \quad 2B \stackrel[k_{-4}]{k_4}{\rlas} 3B
\end{align}
RNDB  requires that $\ds \frac{k_1}{k_{-1}} = \frac{k_2}{k_{-2}}$ and $\ds \frac{k_3}{k_{-3}} = \frac{k_4}{k_{-4}}$ while MCDB requires that $\ds \frac{k_3}{k_{-3}} = \frac{k_4}{k_{-4}}$. Thus the set of conditions required for MCDB is a nonempty, proper subset of the conditions required for RNDB.

These examples illustrate the main result stated in Theorem \ref{thm:main}, that reaction network detailed balance implies Markov chain detailed balance but the converse is not true in general. The results obtained in this section are summarized in Table \ref{tab:examples}.

\section{Chemical reaction network theory} \label{sec:intro_CRN}
\subsection{Introduction and basic definitions}

We begin with a review of the notation and basic definitions related to chemical reaction networks. An example of a {\em chemical reaction} is the following:
\begin{align} \label{eq:exreaction}
  X_{1}+ 2X_{2} ~\rightarrow~ X_2 + X_{3} ~
\end{align}
The $X_{i}$ are called chemical {\em species}, and $X_{1}+2X_{2}$ and
$X_2 + X_{3}$ are called chemical {\em complexes.}  For the reaction in \eqref{eq:exreaction}, ${\vect y} := X_{1}+2X_{2}$ is called the {\em reactant complex} and ${\vect y'} := X_2 + X_{3}$ is called the {\em product complex}, so we may rewrite the reaction as ${\vect y} \ra {\vect y'}$. We will find it convenient to think of the complexes as vectors, for instance, we may assign the reactant complex $X_{1}+2X_{2}$ to the vector $(1,2,0)$ and the product complex  $X_2 + X_{3}$ to the vector $(0,1,1)$. In other words, we are identifying the species $X_i$ with the canonical basis vector whose $i$-th component is $1$ and the other components are $0$. We let $s$ denote the total number of species in the network under consideration and we let $r$ represent the number of reactions, each reaction written as  ${\vect y}_i \rightarrow {\vect y}_i'$, for $i \in \{1,2,\dots,r\}$, and ${\vect y}_i, {\vect y}_i' \in \Z^s_{\ge 0}$, with ${\vect y}_i \ne {\vect y}_i'$. We index the entries of a 
complex vector ${\vect y}_i$ by writing ${\vect y}_i = \left( y_{i1}, y_{i2}, \dots, y_{is}\right) \in \Z^s_{\geq 0}$,
and we will call $y_{ij}$ the {\em stoichiometric coefficient} of species $j$ in 
complex ${\vect y}_i$.  
For ease of notation, when there is no need for
enumeration we typically will drop the subscript $i$ from the notation
for the complexes and reactions. We will reserve {\em boldface} fonts for vectors (usually of dimension $s$ -- the components of which refer to different species). In particular, we may sometimes need to use subscripts to denote a finite sequence of vectors, and so $({\vu_1}, \ldots, {\vu_c})$ will denote a sequence of vectors while $(u_1, \ldots, u_s)$ will denote components of a vector ${\vect u}$. The disambiguating notation is only out of abundance of caution, as the context will usually make the type of the object clear. 

The basic definitions and notations in this paper follow those in \cite{joshi2013complete,joshi2012simplifying, joshi2012atoms}; we start by defining chemical reaction networks. 

\begin{definition}   \label{def:crn}
  Let $\SS = \{X_i\}$, $\CC = \{ {\vect y}\},$ and $\RR = \{ {\vect y} \to  {\vect y'} | {\vect y'} \ne {\vect y}\}$ denote finite sets of species, complexes, and reactions, respectively.  The triple
  $\Net$ is called a {\em chemical reaction network} if it satisfies the following:
  \been
  	\item for each complex $ {\vect y} \in \CC$, there exists a reaction in $\RR$ for which $ {\vect y}$ is the reactant complex or $ {\vect y}$ is the product complex, and
	\item for each species $X_i \in \SS$, there exists a complex $ {\vect y} \in \CC$ that contains $X_i$.
  \enen
\end{definition}
For a chemical reaction network $\Net$, unless otherwise specified, we will denote the number of species by $s := \abs{\SS}$, the number of complexes by $n :=\abs{\CC}$ and the number of reactions by $r:= \abs{\RR}$. 
\begin{definition}
\been
\item
${\vect u}$ is the {\em reaction vector} of reaction $(\vy \ra \vy') \in \RR$ if ${\vect u} = {\vect y'} - {\vect y}$. We will say $\vu$ is a {\em reaction vector} if $\vu$ is the reaction vector of some reaction in $\RR$. Let $V(\RR)$ be the set of all reaction vectors, $V(\RR) := \{ {\vect y'} -  {\vect y} |  {\vect y} \ra  {\vect y'} \in \RR \}$. Let $v := \abs{V(\RR)}$ be the number of reaction vectors in $\RR$. 
\item The {\em stoichiometric subspace} of a network is the vector space over $\R$ spanned by the reaction vectors of the network, $\S_\R(\RR):=\text{span} (V(\RR))$. The {\em stoichiometric module} of a network is the $\Z$-module generated by the reaction vectors of the network, denoted by $\S_\Z(\RR)$.
\item Let $\RR({\vu}) := \{\vy \to \vy' | \vy'-\vy = \vu \}$ represent the set of reactions whose reaction vector is ${\vu}$, so that $\cup_{\vu \in V(\RR)} \RR({\vu}) = \RR$. Let $r({\vect u}) : = \abs{\RR({\vect u})}$ denote the cardinality of the set $\RR({\vect u})$, so that $\sum_{\vu \in V(\RR)} r({\vect u}) = r$. Let $\CC(\vu)$ be the set of reactant complexes for the reactions in $\RR(\vu)$ and let $\CC'(\vu)$ be the set of product complexes for the reactions in $\RR(\vu)$. 
\enen
\end{definition}
\begin{definition}
We say that the reaction $(\vy \ra \vy') \in \RR$ is {\em reversible} if $(\vy' \ra \vy) \in \RR$. We say that a chemical reaction network is reversible or that $\RR$ is reversible if all reactions in $\RR$ are reversible.
\end{definition}
\begin{remark} 
We consider only reversible reaction networks in this article, since detailed balance is only defined for reversible networks. 
\end{remark}
\begin{definition}
\begin{enumerate}
\item
For $x \in \Z_{\ge 0}$ and $y \in \Z$, we define the {\em falling factorial} as follows:
\ba
(x)_y := \begin{cases} x(x-1)\cdots (x-y+1) \quad \text{ for } \quad y \ge 1 \\ 
1 \quad \text{ for } \quad y \le 0
 \end{cases}
\ea
\item We say that $\va \ge \vy$ if $\va - \vy \in \Z_{\ge 0}^s$ and $\va > \vy$ if $\va - \vy \in \Z_{> 0}^s$. In particular, $\va > \v0$ means that $\va \in \Z_{> 0}^s$ and $\va \ge \v0$ means that $\va  \in \Z_{ \ge 0}^s$. 
\item For two vectors of the same dimension, $\va = (a_1, a_2, \ldots, a_s) \in \Z_{\ge 0}^s$ and $\vy = (y_1, y_2, \ldots, y_s) \in \Z^s$, we define the {\em power, falling factorial and factorial} as follows:
\been
\item $\ds
(\va)^{\vy} := \prod_{i=1}^s (a_i)^{y_i}  \in \Z_{\ge 0}
$
\item 
$\ds
(\va)_{\vy} := \prod_{i=1}^s (a_i)_{y_i}  \in \Z_{\ge 0}
$
\item $\ds
\va! := \prod_{i=1}^s a_i!  \in \Z_{\ge 0}
$
\enen
\end{enumerate}
\end{definition}

\subsection{Mass-action kinetics} \label{sec:mak}
For a deterministic chemical reaction network, the state of the system at any time is specified by $\vx = (x_1, x_2, \ldots, x_s) \in \R^s$ where $x_i$ are the species concentrations and for a stochastic chemical reaction network, the state of the system is specified by $\va = (a_1, a_2, \ldots, a_s) \in \R^s$ where $a_i$ are the species numbers. In general, the rate at which a reaction $(\vy \to \vy') \in \RR$ occurs is a function of the reactant complex $\vy$, the product complex $\vy'$ and either the species concentration vector $\vx$ in the deterministic case or the species number vector $\va$ in the stochastic case. Furthermore, this reaction rate function is parametrized by reaction rate constants.

\begin{definition} 
The {\em parametrized reaction rate} is defined as $\Gamma_{d, \vk}: \RR \times \R_{\ge 0}^s \to \R_{\ge 0}$ where $\Gamma_{d, \vk}((\vy \to \vy'), \vx)$ is the deterministic reaction rate or as $\Gamma_{s, \vk}: \RR \times \Z_{\ge 0}^s \to \R_{\ge 0}$ where $\Gamma_{s, \vk}((\vy \to \vy'), \va)$ is the stochastic reaction rate. In each case, $\vk$ represents the set of reaction rate constants which parametrize the reaction rates. 
\end{definition}
The choice of a certain class of functions for reaction rates is referred to as specifying kinetics. In this article, we restrict our attention to the common choice of {\em mass-action kinetics}. For mass-action kinetics, a single reaction rate constant $k_{\vy \to \vy'} > 0$ is associated with each reaction $\vy \to \vy' \in \RR$. For a chemical reaction network $\Net$ with $r = \abs{\RR}$, let the (ordered) set of positive {\em mass-action reaction rate constants} be denoted by $\vk = (k_1, k_2, \dots, k_{r}) \in \Rplus^{r} $. For a reaction vector $\vu$, let $\vk(\vu) = (k_1(\vu), k_2 (\vu), \ldots, k_{r(\vu)}(\vu))$ represent the (ordered) set of rate constants for reactions in the set $\RR(\vu)$, so that $\cup_{\vu \in V(\RR)} \vk(\vu) = \vk$ (up to ordering). For reversible networks, we will always assume {\em consistent ordering} by which we mean that $k_j(-\vu)$ is the rate constant of the backward reaction whose forward reaction has the rate constant $k_j(\vu)$. 

When the chemical reaction network is endowed with stochastic mass-action kinetics, the result is a continuous-time Markov chain model. The alternative is to specify deterministic mass-action kinetics which results in a system of ordinary differential equations. Definition \ref{def:mass-action} defines the two types of mass-action kinetics.

\begin{definition} \label{def:mass-action} 
Let $\vu$ be a reaction vector and let $\RR(\vu) = \{\vy_i(\vu) \stackrel[]{k_i(\vu)}{\longrightarrow} \vy_i(\vu) + \vu | 1 \le i \le r(\vu)\}$.  Here $\CC(\vu) = \{{\vy}_1(\vu), {\vy}_2(\vu), \ldots {\vy}_{r(\vu)}(\vu)\}$ are the reactant complexes,  $\CC'(\vu) = \{{\vect y}_1(\vu) + \vu, {\vect y}_2(\vu) + \vu, \ldots {\vect y}_{r(\vu)}(\vu) + \vu \}$ are the product complexes, and $k_i(\vu) := k_{\vy_i(\vu) \to \vy_i(\vu) + \vu}$ is the reaction rate constant for the reaction $(\vy_i(\vu) \to \vy_i(\vu) + \vu) \in \RR(\vu)$. 

\been
\item 
Let $a_i \in \Z_{\ge 0}$ represent the population of the chemical species $X_i$ and let ${\vect a} = (a_1, \ldots, a_s) \in \Z_{\ge 0}^s$ represent the state of the system. A chemical reaction network is said to be endowed with {\em stochastic mass-action kinetics} if for a state ${\va}$ and for all reaction vectors $\vu$ the transition $\va \ra \va + \vu$ happens at rate:
\ba
\rho (\va, \va+\vu) := \sum_{i=1}^{r(\vu)} k_i(\vu) (\va)_{\vy_i(\vu)} \label{eq:stotrans}
\ea
If the system is in state $\va$, the rate at which the populations of the species change is given by
\ba
\sum_{\vu \in V(\RR)} \vu \sum_{i=1}^{r(\vu)} k_i(\vu) (\va)_{\vy_i(\vu)} =: sto(\va) \label{eq:stochrate}
\ea

\item Let $x_i$ represent the concentration of the chemical species $X_i$ and let $\vx = (x_1,\ldots,x_{s}) \in \R_{\ge 0}^s$ represent the state of the system. A chemical reaction network is said to be endowed with {\em deterministic mass-action kinetics} if the time-evolution of $\vx(t)$ is governed by the following system of ordinary differential equations:
\begin{equation}
\frac{d \vx}{dt} = \sum_{\vu \in V(\RR)} \vu \sum_{i=1}^{r(\vu)} k_i(\vu) \vx^{\vy_i(\vu)}  =:  det( \vx)
  \label{eq:detrate}
\end{equation}
\enen
\end{definition}
The striking resemblance between $sto(\va)$ in \eqref{eq:stochrate} and $det(\vx)$ in \eqref{eq:detrate} is not a random coincidence. In fact, the time evolution of \eqref{eq:detrate} closely approximates the dynamical evolution of \eqref{eq:stochrate} for the special case of large population numbers.  

If the initial state for the stochastic mass-action system is $\va_0 \in \Z_{\ge 0}^s$, then for all future times the trajectory is restricted to $\invtPoly_{\va_0} := (\va_0 + \S_\Z) \cap \Z_{\ge 0}^s$. Thus $\invtPoly_{\va_0}$ is the set of states that are accessible from $\va_0$. Similarly in the case of the deterministic system, if the initial position is $\vx_0 \in \R_{\ge 0}^s$, then the trajectory is confined to $\invtPoly_{\vx_0} := (\vx_0+\S_\R) \cap \mathbb{R}^{s}_{\geq 0}$ for all positive time. In the chemical reaction network theory literature, $\invtPoly_{\vx_0}$ is referred to as the stoichiometric compatibility class containing $\vx_0$. 
\begin{definition} \label{def:stoichcompclass}
\begin{enumerate}
\item
For deterministic mass-action kinetics, we refer to $\invtPoly_{\vx_0} := (\vx_0+\S_\R) \cap \mathbb{R}^{s}_{\geq 0}$ as {\em the stoichiometric compatibility class containing $\vx_0$}. 
\item 
For stochastic mass-action kinetics, we refer to $\invtPoly_{\va_0} := (\va_0 + \S_\Z) \cap \Z_{\ge 0}^s$ as {\em the stoichiometric compatibility class containing $\va_0$}. 
\end{enumerate}
\end{definition}
 
\begin{remark} 
In the stochastic setting, the stoichiometric compatibility class containing $\va_0$ is the set of states that are accessible from $\va_0$. In a reversible chemical reaction network, for all initial states $\va_0 \in \Z^s$ the stochastic stoichiometric compatibility class $\invtPoly_{\va_0}$ is irreducible. The reason for this is that if $\va \in \invtPoly_{\va_0}$ then $\va$ is accessible from $\va_0$, on the other hand reversibility of the underlying chemical reaction network implies that $\va_0$ is accessible from $\va$. For a general CRN (not necessarily reversible), $\invtPoly_{\va_0}$ is a subset of the basin of attraction of some irreducible subset of the Markov chain. 
\end{remark}

\subsection{Topological structure of Markov chain arising from CRN}
\begin{figure}[h!]
\centering
\begin{tikzpicture}[scale=2]
\draw[help lines, dashed, line width=1] (0,0) grid (3,3);
\draw [<->, line width=1.5, black] (0,3.5) -- (0,0) -- (3.5,0);
\draw [->, line width=2, red] (1,1) -- (0,2);
\draw [->, line width=2, green] (0,0) -- (1,0);
\draw [->, line width=2, blue] (0,1) -- (0,0);
\node [below, green] at (0.5,0) {$0 \to A$};
\node [left, blue] at (0,0.5) {$B \to 0$};
\node [above right, red] at (0.5,1.5) {$A+B \to 2B$};
\node [below left] at (0,0) {$(0,0)$};
\node [below] at (1,0) {$(1,0)$};
\node [left] at (0,1) {$(0,1)$};
\node [left] at (0,2) {$(0,2)$};
\node [below right] at (1,1) {$(1,1)$};
\end{tikzpicture}
\caption{The minimal transitions in the reaction network $\{0 \lra A, 0 \lra B, A+B \lra 2B\}$. All possible transitions in the graph of the Markov chain are obtained by translating the minimal transitions to the right and above. In other words, if $\{(a_1, \ldots, a_n) \to (b_1, \ldots, b_n)\}$ is the set of minimal transitions then the set of all possible transitions is $\{(a_1, \ldots, a_n) + \Z_{\ge 0}^n \to (b_1, \ldots, b_n) + \Z_{\ge 0}^n\} $.} 
\end{figure}
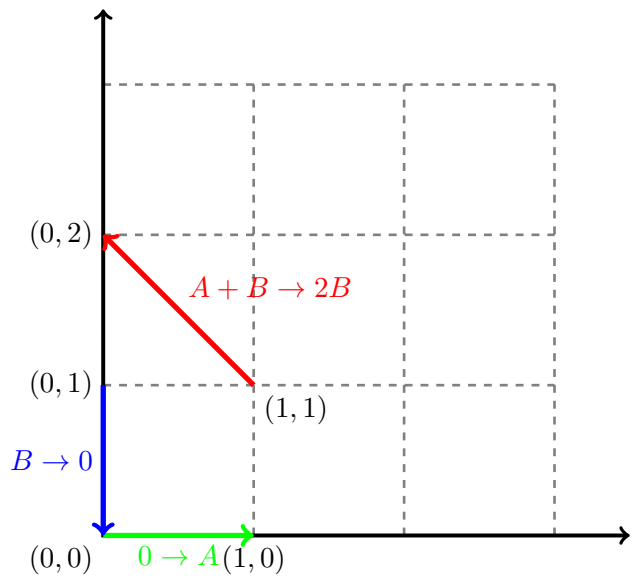

The graph corresponding to the transition (or adjacency) matrix of the Markov chain induced by a chemical reaction network has a special ``lattice structure" which we now describe. 
\begin{definition} {\em $\vu$ is said to be a reaction vector based at $\va$} if there is a $(\vy \ra \vy') \in \RR$ such that $\vy'-\vy = \vu$ and $\va \ge \vy$. 
\end{definition}
\begin{proposition} Suppose that $\vu$ is a reaction vector based at $\va$. Then for ${\vect b} \in \Z_{\ge 0}^s$, $\vu$ is a reaction vector based at $\va + {\vect b}$. 
\end{proposition} 
\begin{proof} Since $\vu$ is a reaction vector based at $\va$, there is at least one reaction $(\vy \ra \vy') \in \RR$ such that $\vy'-\vy = \vu$ and $\va \ge \vy$. Clearly, $\va + {\vect b} \ge \vy$ and so $\vu$ is a reaction vector based at $\va + {\vect b}$. 
\end{proof}
\begin{proposition} If at least one reaction in $\RR(\vu)$ is reversible, and $\vu$ is a reaction vector based at $\va$, then $-\vu$ is a reaction vector based at $\va + \vu$. 
\end{proposition}
\begin{proof}
Let $\vy \ra \vy + \vu$ be reversible, so that $( \vy + \vu \ra (\vy + \vu) - \vu ) \in \RR$, from which the result follows. 
\end{proof}
These easy propositions indicate that the topological structure of the Markov chain that arises from chemical reaction networks is highly constrained, and in particular possesses the highly symmetrical and repeating structure of a lattice.   Many properties of the graph of a Markov chain arising from a CRN are studied in \cite{pauleve2014dynamical}. 
\begin{definition} \been
\item If $\vu$ is a reaction vector based at $\va$ and $-\vu$ is a reaction vector based at $\va + \vu$, then we say that $\vu$ is a {\em reversible reaction vector based at $\va$}. 
\item A cyclically ordered nonempty set of reaction vectors $(\vu_1, \ldots, \vu_c)$  is called a {\em reaction cycle} if $\sum_{i=1}^c \vu_i = {\vect 0}$. By cyclic ordering, we mean that the reaction cycles $(\vu_1, \vu_2, \ldots, \vu_c)$ and $(\vu_2, \vu_3, \ldots, \vu_c, \vu_1)$ are considered identical. 
\beit
\item 
A reaction cycle $U$ is {\em trivial} if $U = (\vu, -\vu)$, otherwise it is {\em nontrivial}. 
\item A reaction cycle $(\vu_1, \ldots, \vu_c)$ is {\em reducible} if there  is a proper consecutive cyclically ordered subsequence of $(\vu_1, \ldots, \vu_c)$ which is a reaction cycle. Otherwise $(\vu_1, \ldots, \vu_c)$ is irreducible. 
\enit
\item A reaction cycle $(\vu_1, \ldots, \vu_c)$ is a {\em reversible reaction cycle} if each $\vu_i$ is a reversible reaction vector. 
\item A {\em reversible reaction cycle $(\vu_1, \ldots, \vu_c; \va)$ based at $\va$} is a reversible reaction cycle $(\vu_1, \ldots, \vu_c)$ along with a state $\va$ such that $\vu_i$ is a reaction vector based at $\va + \sum_{j=1}^{i-1} \vu_j$ and $-\vu_i$ is a reaction vector based at $\va + \sum_{j=1}^{i} \vu_j$ for all $i \in \{1, \ldots, c \}$. 
\enen
\end{definition}
Without explicitly mentioning it, we will always assume modular arithmetic (mod $c$) for the indices in a reaction cycle of length $c$. For instance, for the reaction cycle $(\vu_1, \ldots, \vu_c)$, the indices satisfy $c+1=1$, and the vectors $\vu_c$ and $\vu_1$ are considering adjacent. 

\section{What is detailed balance?} \label{sec:detbal}

\subsection{Detailed balance as a collection of related ideas}

Detailed balance arises in deterministic chemical reaction network theory \cite{ feinberg1989necessary,lewis1925new,  onsager1931reciprocal, wigner1954derivations}, theory of stochastic CRNs \cite{whittle1986systems}, theory of discrete-time and continuous-time Markov chains (see \cite{kelly1979reversibility} for numerous examples), and electrical networks \cite{casimir1949some}. When the context is not provided, a detailed balanced network or a detailed balanced equilibrium should be thought of as a meta-concept or as a collection of related ideas each of which has the following properties: 
\been
\item there is either an explicit or implicit underlying  network structure which is fixed in time, and has the property that for every pair of nodes $\{a,b\}$, if $(a,b)$ is an edge then $(b,a)$ is an edge,
\item there is a positive real-valued function defined on the edge set which is fixed in time, called the edge weight, 
\item there is a dynamically evolving (either in discrete or continuous-time) function, whose evolution is at least partially dictated by the edge weights, called the state variable,
\item a dynamic equilibrium can be defined and exists for the system, 
\item each directed edge has some notion of ``flow" associated with it (which is in general, a function of the edge weight, the node where the edge originates, and the state variable)
\item a detailed balanced equilibrium exists if, at this equilibrium, the flow from node $a$ to node $b$ is equal to the flow from node $b$ to node $a$ for every pair of nodes $(a,b)$ in the network. 
\item detailed balance is a property of the network taken with the edge weights rather than of any particular equilibrium of the dynamics on the network. This is because the conditions that are required to show the detailed balance property of an equilibrium are shown to be equivalent to conditions that can be written as constraints on the fixed edge weights. Such conditions include circuit conditions in the case of deterministic mass-action CRNs, and Kolmogorov cycle conditions in the theory of Markov chains.
\enen

In table \ref{tab:detbal}, we provide the specific instantiations of the above-listed properties in the contexts of (1) chemical reaction network, and (2) graph of Markov chain of the CRN. 

\begin{table}[h!]
\centering
\begin{tabular}{|c || c |c | }
\hline
Network & Chemical reaction network & Graph of MC of CRN \\
\hline
\hline
Vertex set & Chemical complexes & Species population vector \\
\hline
Edge set & Chemical reactions & Positive probability transitions \\
\hline
Size & Finite & Infinite \\
\hline
Edge weight & Reaction rate constant & Transition probability \\
\hline
State variable & Species concentration vector & Probability of state \\
\hline
Equilibrium & Steady state of ODE system & Stationary measure \\
\hline
Flow & Reaction rate & Probability flow \\
\hline
\end{tabular}
\caption{Instances of the ideas related to detailed balance in two specific contexts: (1) chemical reaction network, and (2) graph of Markov chain of the CRN.}
\label{tab:detbal}
\end{table}
We now proceed to define the two types of detailed balance that are relevant for this article. 

\subsection{Detailed balance in deterministic models of CRNs} \label{sec:RNDB}

\begin{definition}[Reaction network detailed balance] \been
\item By {\em steady state of a chemical reaction network}, we mean a steady state of the system of differential equations resulting from applying the kinetic scheme to the network. Thus for deterministic mass-action kinetics, $\vx^*$ is a steady state if $det(\vx^*) = {\vect 0}$ in equation \eqref{eq:detrate}. If $\vx^* \in \R_{>0}^s (\in \R_{\ge 0}^s) $, we say that $\vx^*$ is a {\em positive (non-negative) steady state}. 
\item A steady state of a reversible chemical reaction network is said to be {\em detailed balanced} if at the steady state, the rate of each forward reaction $(\vy \ra \vy') \in \RR$ is equal to the rate of the backward reaction $\vy' \ra \vy$. 
\item A reversible chemical reaction network is said to satisfy {\em reaction network detailed balance} if each positive steady state of the network is detailed balanced. 
\enen
\end{definition}
\begin{remark}
For mass-action kinetics, if one positive steady state of a network is detailed balanced then every positive steady state is detailed balanced. Thus, detailed balance is a property of the network itself rather than any particular steady state of the network, for details see for instance \cite{feinberg1989necessary}. 
\end{remark}

For mass-action kinetics, reaction network detailed balance can be guaranteed by specifying certain relations between the rate constants. Let $\vu$ be a reaction vector. Let \[ 
\left(\vy_i(\vu) \stackrel[]{k_i(\vu)}{\longrightarrow} \vy_i(\vu) + \vu\right) \in \RR(\vu)
\] For a reaction network detailed balanced network, if $\vx^*$ is a positive steady state, then we must have that $k_i(\vu) (\vx^*)^{\vy_i(\vu)} = k_i(- \vu) (\vx^*)^{\vy_i(\vu) + \vu}$ which simplifies to 
\ba
(\vx^*)^{\vu} = \frac{k_i(\vu)}{k_i(- \vu)} \label{eq:RNDBeq}
\ea
This condition which relates the detailed balanced steady state coordinate values to the reaction rate constants can be written simply as a relation between rate constants without invoking the steady states. We state the result in Theorem \ref{thm:RNDBrateconds}. The content of Theorem \ref{thm:RNDBrateconds} can be found in \cite{feinberg1989necessary}. However, we provide an alternate formulation that is suited to our purposes and does not involve the more technical graph-theoretic jargon of \cite{feinberg1989necessary}. We also provide a simple proof for the forward direction. 
\begin{theorem} \label{thm:RNDBrateconds} A reversible chemical reaction network satisfies reaction network detailed balance if and only if the following relations hold between the rate constants: 
\been
\item For a reaction vector $\vu$, 
\ba \label{eq:ratecond1}
\frac{k_1(\vu)}{k_1(-\vu)} = \frac{k_2(\vu)}{k_2(-\vu)} = \ldots = \frac{k_{r(\vu)}(\vu)}{k_{r(\vu)}(-\vu)}
\ea
\item For a reaction cycle $(\vu_1, \vu_2, \ldots, \vu_c)$ and for all $q(\vu_i) \in \{1, 2, \ldots, r(\vu_i)\}$
\ba \label{eq:ratecond2}
\prod_{i=1}^c \frac{k_{q(\vu_i)}(\vu_i)}{k_{q(\vu_i)}(-\vu_i)} = 1
\ea
\enen
\end{theorem}
\begin{proof} Suppose first that a chemical reaction network satisfies reaction network detailed balance. For a reaction cycle $(\vu_1, \vu_2, \ldots, \vu_c)$, $\sum_{i=1}^c \vu_i = \v0$, and so 
\[
1 = (\vx^*)^\v0 = (\vx^*)^{\sum_{i=1}^c  \vu_i} = \prod_{i=1}^c  (\vx^*)^{\vu_i} = \prod_{i=1}^c \frac{k_{q(\vu_i)}(\vu_i)}{k_{q(\vu_i)}(-\vu_i)}
\]
where $1 \le q(\vu_i) \le r(\vu_i)$. We used \eqref{eq:RNDBeq} in the last step thus proving the second identity \eqref{eq:ratecond2}. The first identity \eqref{eq:ratecond1} then follows by considering the reaction cycle $(\vu_1,\vu_2) = (\vu, -\vu)$ and considering $q(\vu_1) \ne q(\vu_2)$. We will omit proof of the converse and refer the interested reader to \cite{feinberg1989necessary}. 
\end{proof}
\begin{remark}
As the proof makes it clear, the second condition for RNDB, \eqref{eq:ratecond2} subsumes the first condition \eqref{eq:ratecond1}. Thus to show RNDB, it is sufficient to show that \eqref{eq:ratecond2} holds for all reversible reaction cycles. Nevertheless, it is convenient to think of \eqref{eq:ratecond1} as a condition on each reaction vector, and \eqref{eq:ratecond2} as a condition on {\em nontrivial reaction cycles}. 
\end{remark}

We now define the notion of complex balancing in chemical reaction networks, first introduced by Horn and Jackson in \cite{horn1972general}. This idea will be useful to us when determining the stationary distribution of certain Markov chains arising from chemical reaction networks. More details on complex balancing can be found in \cite{feinberg1972complex,horn1972necessary}, also see \cite{anderson2010product, dickenstein2011far} for more recent work. Dickenstein and Mill\'{a}n explore the relation between detailed balance and complex balance in reversible networks \cite{dickenstein2011far}. 

\begin{definition} \label{def:complex-balance}
A chemical reaction network is said to be {\em complex balanced} if for each complex $\vy \in \CC$, at equilibrium, the sum of reaction rates of reactions with $\vy$ as the reactant complex is equal to the sum of reaction rates of reactions with $\vy$ as the product complex,
\ba
\sum_{\vy':(\vy \to \vy') \in \RR} \Gamma_{d,\vk}((\vy \to \vy'), \vx^*) = \sum_{\vy':(\vy' \to \vy) \in \RR} \Gamma_{d,\vk}((\vy' \to \vy), \vx^*).
\ea
\end{definition}
It is easy to see that detailed balance implies complex balance. In the case of mass-action kinetics, the complex balancing condition can be written as:
\ba
(\vx^*)^\vy \sum_{\vy':(\vy \to \vy') \in \RR} k_{\vy \to \vy'} ~~ = ~~  \sum_{\vy':(\vy' \to \vy) \in \RR}k_{\vy' \to \vy}(\vx^*)^{\vy'}.
\ea

\subsection{Detailed balance in Markov chains} \label{sec:MCDB}

The definition of Markov chain detailed balance comes from the theory of Markov chains, we will state the definition in generality since it does not depend on the particular setting of chemical reaction networks. 
\begin{definition}[Markov chain detailed balance] \label{def:stochdetbal} A continuous-time Markov chain with state space $\Om$ and transition rate matrix $\rho$ is said to satisfy {\em detailed balance} if there exists a measure  (i.e. a non-negative, countably additive function) $\mu$ on the set of states, such that for every pair of states $x, y \in \Om$, the relation $\mu (x)\rho(x,y) = \mu (y)\rho(y,x)$ holds. A reversible chemical reaction network is said to satisfy {\em Markov chain detailed balance (MCDB)} if the continuous-time Markov chain resulting from applying mass-action kinetics satisfies detailed balance. 
\end{definition}
An equivalent condition to detailed balance in a Markov chain is the Kolmogorov cycle condition, see for instance \cite{durrett2010probability,kelly1979reversibility}. We will make use of this Kolmogorov criterion for establishing Markov chain detailed balance of chemical reaction networks. 
\begin{theorem}[Kolmogorov cycle condition] Consider a Markov chain with state space $\Om$ and an irreducible transition matrix $\rho$. The Markov chain satisfies detailed balance if and only if (i) $\rho(x,y) >0$ implies that $\rho(y,x) >0$ for every pair of states $x$ and $y$, and (ii) for every sequence of states $(x_0=x_n, x_1, x_2, \ldots, x_{n-1}) \subset \Om$ with $\rho(x_i, x_{i+1}) >0$ for $0 \le i \le n-1$, the so-called Kolmogorov cycle condition (KCC) holds:
\ba
\prod_{j=1}^{n}\frac{\rho (x_{j-1}, x_j)}{\rho (x_{j}, x_{j-1})} = 1 \label{eq:kcc}
\ea
\end{theorem}
The transition matrix for a reversible reaction network with the state space restricted to a stoichiometric compatibility class is irreducible. Furthermore, since the reactions are reversible, if $\rho(x,y) >0$ then $\rho(y,x)>0$ for every pair of states $x$ and $y$. Thus, in order to determine conditions for MCDB to hold, we only need to find conditions for KCC to hold. We will now translate KCC to reversible chemical reaction networks with stochastic mass-action kinetics. Consider a reversible reaction cycle $(\vu_1, \vu_2, \ldots, \vu_c)$ based at $\va$. Then $(x_0 = \va, x_1 = \va + \vu_1, x_2 = \va + \vu_1 + \vu_2, \ldots, x_{c-1} = \va + \vu_1 + \vu_2 + \ldots + \vu_{c -1})$ is a cycle of states of length $c$. We replace the transition rates for chemical reaction networks from \eqref{eq:stotrans} into \eqref{eq:kcc} to get the following:
\ba
\prod_{j=1}^{c}\frac{\sum_{i=1}^{r(\vu_j)} k_i (\vu_j) (\va + \vu_1 + \ldots + \vu_{j-1})_{\vy_i(\vu_j)}}{\sum_{i=1}^{r(\vu_j)} k_i (-\vu_j) (\va + \vu_1 + \ldots + \vu_{j})_{\vy_i(-\vu_j)}} = 1 \label{eq:MCDBeq}
\ea
where we used $r(-\vu) = r(\vu)$ which holds for reversible networks. If \eqref{eq:MCDBeq} holds on every cycle $(\vu_1, \vu_2, \ldots, \vu_c)$ in the graph of the transition matrix of the Markov chain, then the network satisfies Markov chain detailed balance. On the face of it, the conditions on the reaction rates required for reaction network detailed balance--\eqref{eq:ratecond1} and \eqref{eq:ratecond2}--look quite dissimilar to the rate condition for Markov chain detailed balance \eqref{eq:MCDBeq}. We will show that the two are closely related, and in fact reaction network detailed balance implies Markov chain detailed balance. Furthermore, in a vast majority of applications the two are equivalent. 

\subsection{Detailed balance in stochastic models of CRNs}
Whittle \cite{whittle1986systems} defined a notion of detailed balance for chemical reaction networks with stochastic mass-action kinetics, which is distinct from the usual definition of detailed balance in Markov chains. In order to avoid confusion, we will refer to Whittle's notion as {\em Whittle stochastic detailed balance (WSDB)}. 
\begin{definition}[Whittle stochastic detailed balance (WSDB)] \label{def:wsdb}
A Markov chain model of a reversible chemical reaction network is said to possess {\em Whittle stochastic detailed balance (WSDB)} if the rate of each forward reaction is equal to the rate of each backward reaction. In other words, for each reaction $\left(\vy_i(\vu) \stackrel[]{k_i(\vu)}{\ra} \vy_i(\vu) + \vu\right) \in \RR(\vu)$, we have that
\ba
k_i(\vu) (\va)_{\vy_i(\vu)} = k_i(-\vu) (\va + \vu)_{\vy_i(\vu) + \vu}
\ea
\end{definition}

\section{Relation between RNDB, WSDB and MCDB} \label{sec:mainthm}

\subsection{Preliminaries}

In this section, we will flesh out the precise relation between RNDB, WSDB and MCDB. First we state a theorem due to Whittle \cite{whittle1986systems}:
\begin{theorem}[Whittle] A reversible chemical reaction network possesses reaction network detailed balanced if and only if it possesses Whittle stochastic detailed balance. 
\end{theorem}
See Whittle \cite{whittle1986systems} for a proof. In this section we will show that in general the following holds for reversible chemical reaction networks with mass-action kinetics:
\ba
\boxed{WSDB \iff RNDB \implies MCDB}
\ea

Let ${\mathcal Z}$ be the set of reversible reaction cycles on $\RR$. Let $\UUU: {\mathcal Z} \times \Z^s \to \R$ be defined via 
\ba \label{eq:functional}
\UUU((\vu_1, \ldots, \vu_c;\va)) := \prod_{j=1}^{c}\frac{\sum_{i=1}^{r(\vu_j)} k_i (\vu_j) (\va + \vu_1 + \ldots + \vu_{j-1})_{\vy_i(\vu_j)}}{\sum_{i=1}^{r(\vu_j)} k_i (-\vu_j) (\va + \vu_1 + \ldots + \vu_{j})_{\vy_i(-\vu_j)}}
\ea
where $(\vu_1 ,\ldots, \vu_c;\va)$ is a reversible reaction cycle based at $\va$. If for all $\va \in \Z_{\ge 0}^s$ and for all reaction cycles $U \in {\mathcal Z}$ based at $\va$, we have $\UUU(U; \va)=1$, then the network satisfies Markov chain detailed balance. In order to find conditions on the rate constants of $\RR$ for Markov chain detailed balance, the first order of business is to bring \eqref{eq:functional} into a more workable form. We start with a lemma which facilitates manipulation of the falling factorials. 
\begin{lemma} \label{lem:calcs}
The following identities hold for $\vx\in \Z_{\ge 0}^s$ and $\va, \vb, \vu \in \Z^s$:
\been[Property 1.]
\item For $\vx \ge \va, \vb$, ~
$ \displaystyle
\frac{(\vx)_\va}{(\vx)_\vb} = \frac{(\vx-\vb)_{\va-\vb}}{(\vx-\va)_{\vb-\va}} 
$. \\ In particular if $\vx \ge \va \ge \vb$, then $ \displaystyle
\frac{(\vx)_\va}{(\vx)_\vb} = (\vx-\vb)_{\va-\vb}
$.
\item 
For $\vx \ge \va$ and $\vx + \vu \ge {\vect 0}$, $ \displaystyle
\frac{(\vx)_\va}{(\vx+\vu)_{\va+\vu}} = \frac{(\vx)_{-\vu}}{(\vx+\vu)_{\vu}} =  \frac{\vx!}{(\vx+\vu)!} 
$.
\enen
\end{lemma}
\begin{proof} 
If $\vx = \va$, then $\va \ge \vb$, so that $(\vx - \va)_{\vb-\va} = 1$. Similarly, if $\vx = \vb$ then $(\vx - \vb)_{\va-\vb} = 1$. In particular, this implies that for $\vx \ge \va, \vb$, we have $(\vx - \va)_{\vb-\va} \ge 1$ and $(\vx - \vb)_{\va-\vb} \ge 1$. 
So that for $1 \le i \le s$, 
\begin{align*}
\frac{(x_i)_{a_i}}{(x_i)_{b_i}} &= \frac{(x_i)(x_i-1) \ldots (x_i-a_i+1)}{(x_i)(x_i-1) \ldots (x_i-b_i+1)}  \\
&= \begin{cases} (x_i - b_i)_{a_i - b_i} \quad \mbox{ if } \quad a_i \ge b_i\\
\frac{1}{(x_i - a_i)_{b_i - a_i}} \quad \mbox{ if } \quad  a_i < b_i \end{cases} \\
&= \frac{(x_i - b_i)_{a_i - b_i}}{(x_i - a_i)_{b_i - a_i}}
\end{align*}
Then Property 1 follows by taking a product over $1 \le i \le  s$. 

Since $\vx \ge \va$ and $\vx + \vu \ge {\vect 0}$, $(\vx)_\va > 0$ and $(\vx+\vu)_{\va+\vu} > 0$. For $1 \le i \le s$, 
\begin{align*}
\frac{(x_i)_{a_i}}{(x_i+u_i)_{a_i+u_i}} &= \frac{(x_i)(x_i-1) \ldots (x_i-a_i+1)}{(x_i+u_i)(x_i+u_i-1) \ldots (x_i-a_i+1)} \\
&= \begin{cases} \frac{1}{(x_i+u_i)_{u_i}} \quad \mbox{ if } \quad u_i \ge 0 \\
(x_i)_{-u_i} \quad \mbox{ if } \quad u_i < 0 
\end{cases} \\
&= \frac{(x_i)_{-u_i}}{(x_i+u_i)_{u_i}} \\
&=  \frac{x_i!}{(x_i+u_i)!}
\end{align*}
The last equality can be checked to hold both for $u_i \ge 0$ and $u_i < 0$ when $x_i +u_i >0$. Both equations of Property 2 then follow by taking a product over $1 \le i \le s$. 
\end{proof}
Recall the definition of stochastic mass-action kinetics which gives the rate of transition $\va \ra \va + \vu$ for a reaction vector $\vu$:
\ba
\rho (\va, \va+\vu) = \sum_{i=1}^{r(\vu)} k_i(\vu) (\va)_{\vy_i(\vu)} 
\ea
Further recall that $y_{ij}(\vu)$ is the stoichiometric coefficient of species $j$ in the reactant complex $\vy_i(\vu)$. Let $m_j(\vu) := \min \{ y_{ij} (\vu) | 1 \le i \le r(\vu) \}$ be the smallest stoichiometric coefficient of species $j$ amongst all reactant complexes corresponding to the reaction vector $\vu$. Let $\vm(\vu) : = (m_1(\vu), m_2(\vu), \ldots, m_s(\vu))$. 
\begin{lemma} \label{lem:back} Let $\vu$ be a reversible reaction vector based at $\va$. Then:
\been
\item $\vy_i(-\vu) = \vy_i(\vu) + \vu$. 
\item $\vm(-\vu) = \vm(\vu) + \vu$. 
\enen
\end{lemma}
\begin{proof}
If $\vy_i(\vu)$ is a reactant complex then $(\vy_i(\vu) \ra \vy_i(\vu) + \vu) \in \RR(\vu)$. This implies that $(\vy_i(\vu) + \vu \ra \vy_i(\vu)) \in \RR(-\vu)$. Since $\RR(\vu)$ and $\RR(-\vu)$ are consistently ordered, it then follows that $\vy_i(-\vu) = \vy_i(\vu) + \vu$. 

From the definition of $m_j(\vu)$, $m_j(-\vu) =  \min \{ y_{ij} (-\vu) | 1 \le i \le r(\vu) \} = \min \{ y_{ij} (\vu) + u_j | 1 \le i \le r(\vu) \} =  \min \{ y_{ij} (\vu) | 1 \le i \le r(\vu) \} + u_j$. The result then follows.
\end{proof}

\begin{lemma} \label{lem:extract} If $\vu$ is a reaction vector based at $\va$, then the rate of transition $\va \to \va + \vu$ is
\ba
\rho (\va, \va+\vu) = (\va)_{\vm(\vu)} \sum_{i=1}^{r(\vu)} k_i(\vu) (\va -\vm(\vu))_{\vy_i(\vu) -\vm(\vu)} 
\ea
\end{lemma}
\begin{proof} By definition of $\vm(\vu)$, $\vy_i(\vu) \ge \vm(\vu)$ for all reaction vectors $\vu$. Since $\vu$ is a reaction vector based at $\va$, there is at least one $\vy_i(\vu)$ such that $\va \ge \vy_i(\vu)$. So that 
\begin{align*}
\rho (\va, \va+\vu) &= \sum_{i=1}^{r(\vu)} k_i(\vu) (\va)_{\vy_i(\vu)}  = (\va)_{\vm(\vu)} \sum_{i=1}^{r(\vu)} k_i(\vu)\frac{(\va)_{\vy_i(\vu)}}{(\va)_{\vm(\vu)}} \\
& = (\va)_{\vm(\vu)} \sum_{i=1}^{r(\vu)} k_i(\vu) (\va -\vm(\vu))_{\vy_i(\vu) -\vm(\vu)} 
\end{align*}
where in the last step we used Property 1 of Lemma \ref{lem:calcs}. 
\end{proof}
\begin{lemma} \label{lem:fact} If $\vu$ is a reversible reaction vector based at $\va$ then 
\ba
\frac{(\va)_{\vm(\vu)}}{(\va + \vu)_{\vm(-\vu)}} = \frac{\va!}{(\va + \vu)!}
\ea
\end{lemma}
\begin{proof}
By Lemma \ref{lem:back}, $\vm(-\vu) = \vm(\vu) + \vu$. Since $\vu$ is a reaction vector based at $\va$, there is at least one $\vy_i(\vu)$ such that $\va \ge \vy_i(\vu) \ge \vm(\vu)$. Further since $\vu$ is a reversible reaction vector based at $\va$, it must be that $-\vu$ is a reaction vector based at $\va + \vu$ and so $\va + \vu \ge \v0$. The result then follows immediately from Property 2 of Lemma \ref{lem:calcs}. 
\end{proof}
\begin{lemma} \label{lem:prodone} For a reversible reaction cycle $(\vu_1, \vu_2, \ldots, \vu_{c})$ based at $\va$, 
\ba
\prod_{j=1}^c \frac{(\va+ \vu_1 + \ldots + \vu_{j-1})_{\vm(\vu_j)}}{(\va+ \vu_1 + \ldots + \vu_{j})_{\vm(-\vu_j)}} = 1
\ea
\end{lemma}
\begin{proof} By definition of a reversible reaction cycle based at $\va$, for all $1 \le j \le c$, $\vu_j$ is a reversible reaction vector based at $\va+\vu_1+ \ldots + \vu_{j-1}$. So by Lemma \ref{lem:fact}, 
\[
\prod_{j=1}^c \frac{(\va+ \vu_1 + \ldots + \vu_{j-1})_{\vm(\vu_j)}}{(\va+ \vu_1 + \ldots + \vu_{j})_{\vm(-\vu_j)}} = \prod_{j=1}^c \frac{(\va+ \vu_1 + \ldots + \vu_{j-1})!}{(\va+ \vu_1 + \ldots + \vu_{j})!}  = 1. 
\]
\end{proof}
\begin{lemma} \label{lem:simpuuu} For a reversible reaction cycle $(\vu_1, \ldots, \vu_c)$ based at $\va$
\ba \label{eq:simpuuu}
\UUU((\vu_1, \ldots, \vu_c; \va)) = \prod_{j=1}^{c}\frac{\sum_{i=1}^{r(\vu_j)} k_i (\vu_j) (\va + \vu_1 + \ldots + \vu_{j-1} - \vm(\vu_j))_{\vy_i(\vu_j) - \vm(\vu_j)}}{\sum_{i=1}^{r(\vu_j)} k_i (-\vu_j) (\va + \vu_1 + \ldots + \vu_{j-1} - \vm(\vu_j))_{\vy_i(\vu_j) - \vm(\vu_j)}}
\ea
\end{lemma}
\begin{proof} By definition of $\UUU$, 
\begin{align*}
& \UUU((\vu_1, \ldots, \vu_c; \va)) = \prod_{j=1}^{c}\frac{\sum_{i=1}^{r(\vu_j)} k_i (\vu_j) (\va + \vu_1 + \ldots + \vu_{j-1} )_{\vy_i(\vu_j) }}{\sum_{i=1}^{r(\vu_j)} k_i (-\vu_j) (\va + \vu_1 + \ldots + \vu_{j} )_{\vy_i(-\vu_j) }} \\
&= \prod_{j=1}^{c}\frac{(\va + \vu_1 + \ldots + \vu_{j-1})_{\vm(\vu_j)}\sum_{i=1}^{r(\vu_j)} k_i (\vu_j) (\va + \vu_1 + \ldots + \vu_{j-1} -\vm(\vu_j))_{\vy_i(\vu_j) -\vm(\vu_j) }}{(\va + \vu_1 + \ldots + \vu_{j})_{\vm(-\vu_j)}\sum_{i=1}^{r(\vu_j)} k_i (-\vu_j) (\va + \vu_1 + \ldots + \vu_{j} -\vm(-\vu_j))_{\vy_i(-\vu_j) -\vm(-\vu_j) }} \\
&= \prod_{j=1}^{c}\frac{\sum_{i=1}^{r(\vu_j)} k_i (\vu_j) (\va + \vu_1 + \ldots + \vu_{j-1} -\vm(\vu_j))_{\vy_i(\vu_j) -\vm(\vu_j) }}{\sum_{i=1}^{r(\vu_j)} k_i (-\vu_j) (\va + \vu_1 + \ldots + \vu_{j} -\vm(-\vu_j))_{\vy_i(-\vu_j) -\vm(-\vu_j) }} \\
&= \prod_{j=1}^{c}\frac{\sum_{i=1}^{r(\vu_j)} k_i (\vu_j) (\va + \vu_1 + \ldots + \vu_{j-1} - \vm(\vu_j))_{\vy_i(\vu_j) - \vm(\vu_j)}}{\sum_{i=1}^{r(\vu_j)} k_i (-\vu_j) (\va + \vu_1 + \ldots + \vu_{j-1} - \vm(\vu_j))_{\vy_i(\vu_j) - \vm(\vu_j)}}
\end{align*}
The first line is equation \eqref{eq:functional}, the second line follows from Lemma \ref{lem:extract}, the third line from Lemma \ref{lem:prodone} and the last line from Lemma \ref{lem:back}.  
\end{proof}

We introduce some notation which will simplify the appearance of $\UUU((\vu_1, \ldots, \vu_c; \va))$. 

\beit
\item To each reaction vector $\vu$ based at $\va$ we associate a vector of length $r(\vu)$ via the following definition:
\ba
& \FFF(\vu, \va) : =  \left<\fff_1(\vu,\va), \fff_2(\vu,\va), \ldots, \fff_{r(\vu)}(\vu,\va)  \right> \nonumber \\
& := \left< (\va - \vm(\vu))_{\vy_1(\vu) - \vm(\vu)}, (\va - \vm(\vu))_{\vy_2(\vu) - \vm(\vu)}, \ldots, (\va - \vm(\vu))_{\vy_{r(\vu)}(\vu) - \vm(\vu)} \right>
\ea
We state some simple but useful properties of $\FFF$ as a lemma.
\begin{lemma} \label{lem:propsofFFF}
For a reversible reaction vector $\vu$ based at $\va$, 
\been[(i)]
\item $\FFF(-\vu, \va+\vu) = \FFF(\vu,\va)$.  
\item If $r(\vu) = r(-\vu) =1$, then $\FFF(\vu,\va) = \left<1\right> $ for all $\va \in \Z_{\ge 0}^s$. 
\enen
\end{lemma}
\begin{proof}
The first result follows easily from Lemma \ref{lem:back} and the second result follows from the fact that if $r(\vu)=1$ then $\vy_1(\vu) = \vm(\vu)$.
\end{proof}

\item Let $\va_j := \va + \sum_{k=1}^{j} \vu_k$. So that to each reaction cycle $(\vu_1, \vu_2, \ldots, \vu_c)$ based at $\va$ we associate the sequence $(\va_0, \va_1, \va_2, \ldots, \va_{c-1}) = (\va, \va + \vu_1, \va + \vu_1 + \vu_2, \ldots , \va + \vu_1 + \ldots + \vu_{c-1})$. 
\item These definitions result in the following compactified version of \eqref{eq:simpuuu}
\ba \label{eq:shortuuu}
\UUU((\vu_1, \ldots, \vu_c; \va)) = \prod_{j=1}^{c}\frac{\sum_{i=1}^{r(\vu_j)} k_i (\vu_j)\fff_i(\vu_j, \va_{j-1})}{\sum_{i=1}^{r(\vu_j)} k_i (- \vu_j)\fff_i(\vu_j, \va_{j-1})} = \prod_{j=1}^{c} \frac{\vk(\vu_j) \cdot \FFF(\vu_j, \va_{j-1})}{\vk(-\vu_j) \cdot \FFF(\vu_j, \va_{j-1})}
\ea
\enit
Here the symbol ``~$\cdot$~" denotes the usual scalar product. Thus the condition for MCDB is that for every reversible reaction cycle $(\vu_1, \vu_2, \ldots, \vu_c)$ based at $\va$, we must have that: 
\ba \label{eq:MCDBcond}
\UUU((\vu_1, \ldots, \vu_c; \va))  = \prod_{j=1}^{c} \frac{\vk(\vu_j) \cdot \FFF(\vu_j, \va_{j-1})}{\vk(-\vu_j) \cdot \FFF(\vu_j, \va_{j-1})} = 1
\ea
where $\FFF(\vu, \va) = \left< (\va - \vm(\vu))_{\vy_1(\vu) - \vm(\vu)},  \ldots, (\va - \vm(\vu))_{\vy_{r(\vu)}(\vu) - \vm(\vu)} \right>$. It is clear that in order to show MCDB, it suffices to check conditions \eqref{eq:MCDBcond} for all irreducible, nontrivial, reversible reaction cycles. 

%
%

\subsection{RNDB implies MCDB}

We now possess the technology required to prove the main result of the article which is that RNDB implies MCDB.

\begin{theorem} \label{thm:main}
Suppose that a reversible reaction network $(\SS, \CC, \RR)$ satisfies reaction network detailed balance. Then $(\SS, \CC, \RR)$ satisfies Markov chain detailed balance.
\end{theorem}
\begin{proof} 
Let $(\vu_1,\vu_2, \ldots, \vu_c)$ be a reaction cycle based at $\va$. For $1 \le j \le c$, 
\ba
\frac{ \vk (\vu_j)\cdot \FFF(\vu_j,\va_{j-1})}{\vk (- \vu_j)\cdot \FFF(\vu_j,\va_{j-1})} = \frac{k_{1}(\vu_j)}{k_{1}(-\vu_j)}\frac{\frac{\vk (\vu_j)}{k_{1}(\vu_j)}\cdot \FFF(\vu_j,\va_{j-1})}{\frac{\vk (- \vu_j)}{k_{1}(- \vu_j)} \cdot \FFF(\vu_j,\va_{j-1})} =  \frac{k_{1}(\vu_j)}{k_{1}(-\vu_j)}. 
\ea
where we used RNDB condition \eqref{eq:ratecond1} from Theorem \ref{thm:RNDBrateconds} to get the last equality above.  
\ba
\UUU((\vu_1, \ldots, \vu_c; \va)) = \prod_{j=1}^{c}\frac{k_1 (\vu_j)}{k_1 (-\vu_j)} = 1
\ea
where the last equality follows from the RNDB condition \eqref{eq:ratecond2}. 
\end{proof}

Thus reaction network detailed balance in a reversible chemical reaction network endowed with mass-action kinetics implies Markov chain detailed balance. The converse is not true in general. We study some instances where the converse holds, followed by instances where the converse does not hold. 

\subsection{When does MCDB imply RNDB?}

Even though MCDB does not imply RNDB in general, the two are in fact equivalent for a vast majority of networks. There is at least one commonly occurring case where the conditions for the equivalence can be checked at a glance. This occurs when there is exactly one reaction corresponding to each reaction vector. We state the result as Theorem \ref{thm:oneveconerxn}. 

\begin{theorem}  \label{thm:oneveconerxn}
Suppose that $r(\vu) =1$ for all $\vu \in V(\RR)$ in a reversible chemical reaction network $(\SS,\CC,\RR)$. Then $(\SS,\CC,\RR)$ has reaction network detailed balance if and only if it has Markov chain detailed balance. 
\end{theorem}
\begin{proof}
Consider a reaction cycle $(\vu_1, \vu_2, \ldots, \vu_c)$ based at $\va$. Substituting $r(\vu_j) =1$ in \eqref{eq:shortuuu} and using Lemma \ref{lem:propsofFFF}, we find that 
\begin{align*}
\UUU((\vu_1, \ldots, \vu_c; \va)) = \prod_{j=1}^{c}\frac{ k_i (\vu_j)}{ k_i (- \vu_j)}
\end{align*}
which is equal to $1$ if and only if RNDB is satisfied (by RNDB condition \eqref{eq:ratecond2}). This shows that RNDB and MCDB are equivalent. 
\end{proof}

\subsection{Which networks have MCDB but do not have RNDB?}

In a reversible reaction network, if there is only one pair of reversible reaction vectors $(\vu, -\vu)$ and there is more than one reaction whose reaction vector is $\vu$, MCDB holds independent of any constraints, but RNDB fails to hold in general. 


\begin{theorem} \label{thm:b&d}
Consider a reversible reaction network $(\SS,\CC,\RR)$ with $V(\RR) = \{\vu, -\vu\}$. Then $(\SS,\CC,\RR)$ satisfies Markov chain detailed balance for all reaction rate constants. 
\end{theorem}
\begin{proof} From any state $\va$, the set of states that can be reached in one time step is a subset of $\{ \va + \vu, \va - \vu \}$, so the process is a reversible birth and death process which is known to satisfy MCDB. 
\end{proof} 
Network 2 in Section \ref{sec:motexamples} provides an example of a network which satisfies the hypotheses of Theorem \ref{thm:b&d}. Another example is the network in \eqref{eq:onerxnatom} in Section \ref{sec:b&d}. For birth and death processes, MCDB holds because the absence of cycles in the graph of the Markov chain makes the Kolmogorov cycle conditions hold vacuously. However, as Theorem \ref{thm:linindV} shows, the mere presence of cycles is not sufficient to make MCDB equivalent to RNDB. 
\begin{theorem} \label{thm:linindV}
Consider a reversible reaction network $(\SS,\CC,\RR)$ such that for all $\vu \in V(\RR)$
\been[(a)]
\item $\vu \notin  \rm{span}(V(\RR) \setminus \{\vu, -\vu \})$. 
\item If $\vv \in  \rm{span}(V(\RR) \setminus \{\vu, -\vu \})$, then $\FFF(\vu,\va) = \FFF(\vu,\va+\vv)$.  
\enen
Then $(\SS,\CC,\RR)$ satisfies MCDB for all reaction rate constants. 
\end{theorem}
We need the following lemma in order to prove Theorem \ref{thm:linindV}. 
\begin{lemma} \label{lem:setup}
Consider a reversible reaction network $(\SS,\CC,\RR)$. Suppose that $\vu \in V(\RR)$ is a reaction vector with the following properties: 
\been[(a)]
\item $\vu \notin  \rm{span}(V(\RR) \setminus \{\vu, -\vu \})$. 
\item If $\vv \in  \rm{span}(V(\RR) \setminus \{\vu, -\vu \})$, then $\FFF(\vu,\va) = \FFF(\vu,\va+\vv)$.  
\enen
Then MCDB does not require any constraints on $\vk(\vu)$ or $\vk(-\vu)$.
\end{lemma}
\begin{proof} Let $\vu$ be a reaction vector that satisfies the hypotheses of the theorem. Let $U := (\vu_1=\vu, \vu_2, \ldots, \vu_c)$ be an irreducible, nontrivial, reversible reaction cycle based at $\va$. Since $\vu \notin  \rm{span}(V(\RR) \setminus \{\vu, -\vu \})$, we must have that $-\vu \in U$, so that $U = (\vu, \vu_2, \ldots, \vu_k, -\vu, \vu_{k+2}, \ldots, \vu_c)$. Since $U$ is irreducible, $\vu$ and $-\vu$ are nonconsecutive and $\vv := \sum_{i=2}^k \vu_i$ is nonzero and in $\rm{span}(V(\RR) \setminus \{\vu, -\vu \})$. So we can write the reaction cycle $U$ as $U= (\vu, \vv, -\vu, -\vv)$, where $\vv$ is a sum of reaction vectors. 

In the expression of $\UUU(U;\va)$ the factor involving $\vk(\vu)$ and $\vk(-\vu)$ is:
\begin{align*}
 \frac{\vk(\vu) \cdot \FFF(\vu, \va)}{\vk(-\vu) \cdot \FFF(\vu, \va)} \frac{\vk(-\vu) \cdot \FFF(-\vu, \va + \vu + \vv)}{\vk(\vu) \cdot \FFF(-\vu, \va +\vu+ \vv)} 
=   \frac{\vk(\vu) \cdot \FFF(\vu, \va)}{\vk(-\vu) \cdot \FFF(\vu, \va)} \frac{\vk(-\vu) \cdot \FFF(\vu, \va + \vv)}{\vk(\vu) \cdot \FFF(\vu, \va + \vv)} = 1
\end{align*}
where the first equality follows from applying Lemma \ref{lem:propsofFFF}, and the next equality follows from $\FFF(\vu,\va) = \FFF(\vu,\va+\vv)$. Since the other factors in $\UUU(U;\va)$ do not involve either $\vk(\vu)$ or $\vk(-\vu)$, it is clear that any set of conditions for MCDB do not constrain $\vk(\vu)$ or $\vk(-\vu)$. 
\end{proof}

\begin{theorem} \label{thm:MCDBwithoutRNDB}
Consider a reversible reaction network $(\SS,\CC,\RR)$. Suppose that $\vu \in V(\RR)$ is a reaction vector with the following properties: 
\been[(a)]
\item $\vu \notin  \rm{span}(V(\RR) \setminus \{\vu, -\vu \})$. 
\item If $\vv \in  \rm{span}(V(\RR) \setminus \{\vu, -\vu \})$, then $\FFF(\vu,\va) = \FFF(\vu,\va+\vv)$.  
\item $r(\vu) >1$. 
\enen
Then there exist a set of rate constants for which the reaction network possesses MCDB but does not possess RNDB. 
\end{theorem}

\begin{proof} By Lemma \ref{lem:setup}, $\vk(\vu)$ and $\vk(-\vu)$ are unconstrained. On the other hand, if RNDB holds then $\vk(-\vu)$ is a constant multiple of $\vk(\vu)$. Since $r(\vu) >1$, this results in a nontrivial condition on the rate constants, a condition absent from the requirements for MCDB, thus proving the theorem.
\end{proof}
In the following example, we look at a network with two reaction vector pairs $(\pm 1,0)$ and $(0, \pm 1)$. The conditions of theorem \ref{thm:MCDBwithoutRNDB} hold for $(\pm 1,0)$ but not for $(0, \pm 1)$, and so we find that the MCDB constraints form a proper subset of RNDB constraints. 
\begin{example}
Consider the following example Network 4 from section \ref{sec:motexamples}. 
\begin{align*}
0 \stackrel[k_{-1}]{k_1}{\rlas} A \quad, \quad 2A \stackrel[k_{-2}]{k_2}{\rlas} 3A  \quad, \quad
A \stackrel[k_{-3}]{k_3}{\rlas} A+B \quad, \quad 2B \stackrel[k_{-4}]{k_4}{\rlas} 3B
\end{align*}
Let $\vu = (1,0)$ and $\vv=(0,1)$. It is easy to show that $\FFF(\vu,\va) = \left< 1, a(a-1)\right>$, so that $\FFF(\vu,\va) = \FFF(\vu,\va+\vv)$. Clearly, $\vu$ and $\vv$ are linearly independent, so $\vu$ satisfies the hypotheses of Theorem \ref{thm:MCDBwithoutRNDB}, and thus MCDB does not require any constraints on $\vk(\vu) = \{k_1, k_2\}$ and $\vk(-\vu) = \{k_{-1}, k_{-2}\}$. RNDB does require that $\frac{k_1}{k_{-1}} = \frac{k_2}{k_{-2}}$. Note however that $\FFF(\vv,\va) = \left< a, b(b-1)\right>$, and so $\FFF(\vv,\va) \ne \FFF(\vv,\va + \vu)$. Thus MCDB requires that $\frac{\vk(\vv)}{k_1(\vu)} = \frac{\vk(-\vv)}{k_1(-\vu)}$ or $\frac{k_3}{k_4} = \frac{k_{-3}}{k_{-4}}$ which is also necessary for RNDB. 
\end{example}
Now we are ready to prove Theorem \ref{thm:linindV}. 
\begin{proof}[Proof of Theorem \ref{thm:linindV}] Let $\vu \in V(\RR)$. Both hypotheses of Lemma \ref{lem:setup} are satisfied for $\vu$, thus implying that there are no constraints on $\vk(\vu)$ or $\vk(-\vu)$. Since this is true for all reaction vectors $\vu \in V(\RR)$, in fact no conditions are required on the rate constants for MCDB. 
\end{proof}

The following example illustrates the content of Theorem \ref{thm:linindV} and Theorem \ref{thm:MCDBwithoutRNDB} that having cycles in the graph of the MC is not sufficient to make RNDB and MCDB equivalent. 
\begin{example}
Consider a ``disjoint union" of two networks: 
\ba 
0 \stackrel[k_{-1}]{k_1}{\rlas} A \quad, \quad 2A \stackrel[k_{-2}]{k_2}{\rlas} 3A \quad, \quad 0 \stackrel[k_{-3}]{k_3}{\rlas} B \quad, \quad 2B \stackrel[k_{-4}]{k_4}{\rlas} 3B
\ea
Notice that the network decouples into a subnetwork involving only species $A$ and a subnetwork involving only species $B$. (The shared $0$ does not count as a species.) The network is easily shown to satisfy the hypotheses of Theorem \ref{thm:linindV}, thus the network satisfies MCDB for all reaction rate constants. However, RNDB is satisfied only when certain constraints on the rate parameters are satisfied.
\end{example}
\section{Applications} \label{sec:applications}

\subsection{A new algorithm for determining MCDB conditions}

Since the conditions on rate constants for RNDB guarantee MCDB, there exists a set of conditions for MCDB which is a subset of the conditions for RNDB. For a complicated reaction network, there may be a huge number of cycle types in the graph of the corresponding Markov chain, which means that finding the complete set of conditions for MCDB may be nontrivial. However, we can first determine the conditions for RNDB, which is a significantly simpler exercise, and then use these conditions to circumscribe the set of conditions for MCDB. 

As a demonstration of this technique, consider the CRN presented in figure \ref{fig:crn_ex} of section \ref{sec:introduction}. We were quite easily able to determine the conditions for RNDB given in equations \eqref{eq:crn_ex_cons1} and \eqref{eq:crn_ex_cons2}. Thus we know that the conditions for MCDB is a subset of these two conditions. By applying Kolmogorov cycle condition to the two cycle types depicted in figure \ref{fig:crn_ex_mc}, we quickly rediscover the two conditions in equations \eqref{eq:crn_ex_cons1} and \eqref{eq:crn_ex_cons2}. Even though we have not examined all possible cycle types in the MC, we know we have obtained a complete list of MCDB conditions because MCDB cannot have more constraints than RNDB. 

The observation in the previous paragraphs allows us to state the following algorithm for determining the conditions on the rate constants for MCDB. Let $(\SS,\CC,\RR)$ be a reversible reaction network. We will represent the set of constraints on $(\SS,\CC,\RR)$ for RNDB by $\TD$. 
\paragraph{Algorithm for determining the set of constraints on the rate constants that result in detailed balance in the MC model of the CRN:} ~\\

\noindent {\bf Input:} A reversible chemical reaction network $(\SS, \CC, \RR)$ with mass-action kinetics.

\noindent {\bf Output:} A set of conditions on reaction rate constants of $(\SS, \CC, \RR)$ which results in MCDB. 
\been
\item There exists a set of constraints for MCDB which is contained in $\TD$. Denote this set by $\TS \subseteq \TD$. The next steps construct $\TS$. 
\item Assume RNDB holds on the CRN and obtain all constraints on the rate constants. This will give our initial set $S_0 = \TD$. Let $T_0 = \emptyset$ initially. We will successively select elements of $S_0$ and either discard them or move them to $T_0$. The algorithm terminates when $S_0$ is empty. 
\item If $S_0$ is empty then go to step 6. Otherwise, let $z \in S_0$. $z$ is a relation involving a subset $\vk_z$ of rate constants $\vk = \cup_{\vu \in V(\RR)} \vk({\vu})$. 
\item Let $C_z$ be the set of irreducible, nontrivial, reversible cycle types in the graph of the MC that involves all the rate constants in $\vk_z$.  Let $L_0 = C_z$. 
\item If $L_0$ is empty, then discard $z$ from $S_0$ and go to step 3. Otherwise, select a cycle type from $L_0$, and apply Kolmogorov cycle condition (KCC) to this cycle. If this results in the constraint $z$, then move $z$ from $S_0$ to $T_0$ and go to step 3. If not, then discard $z$ and go back to the beginning of this step. 
\item Let $\TS = T_0$. 
\enen

\subsection{Stationary distribution of a network with RNDB} \label{sec:statdist}

Theorem 4.1 in \cite{anderson2010product} provides the stationary distribution of the Markov chain arising from a complex balanced reaction network (see Definition \ref{def:complex-balance}), where the explicit formula for the stationary distribution is in terms of the steady state of the corresponding deterministic model. A detailed balanced CRN is also complex balanced and thus the above-mentioned formula provides the stationary distribution of the Markov chain arising from a detailed balanced CRN as well. However, detailed balance allows a further simplification -- we are not required to find a steady state of the deterministic system. 

For a reaction vector $\vu \in V(\RR)$, let $k(\vu) := \frac{k_1(\vu)}{k_{1}(-\vu)}$. Then for a reversible network with RNDB, we have that $(\vx^*)^\vu = k(\vu) = \frac{k_1(\vu)}{k_{1}(-\vu)} = \frac{k_1(\vu)}{k_{1}(-\vu)} = \ldots = \frac{k_{r(\vu)}(\vu)}{k_{r(\vu)}(-\vu)}$. Let $V(\RR) := \{\vu_1, \vu_2, \ldots, \vu_v \}$. Let $\vk = (k(\vu_1), k(\vu_2), \ldots, k(\vu_v))$. For $t \ge 0$ if $\va(t)$ is the state of the Markov chain at time $t$, then there exists an ordered set $\valpha = (\alpha_1, \alpha_2, \ldots, \alpha_v)   \in \Z^v$ such that $\va(t) = \va(0) + \sum_{i=1}^v \alpha_i \vu_i$. For a pair of states $(\va, \vb)$ define $\rm{comp}_\vu(\va; \vb) :=  \valpha = (\alpha_1, \alpha_2, \ldots, \alpha_v) \in \Z^v$ to be any ordered set which satisfies $\va - \vb = \sum_{i=1}^v \alpha_i \vu_i$. Note that no uniqueness is being claimed for $\rm{comp}_\vu(\va; \vb)$.

For a state $\va$ in the Markov chain, let $\Lambda_\va$ be the set of states that communicate with $\va$. Let $\va_0 \in \Lambda_\va$ be an arbitrary reference state. Then we will show that the stationary distribution $\pi$ is given by 
\ba
\pi(\va) \propto \frac{1}{\va!} \vk^{\rm{comp}_\vu(\va; \va_0)}
\ea

The stationary distribution can be directly constructed from the reaction network detailed balance property of the chemical reaction network. Alternatively, one can use the following theorem of Anderson {\it et al}. (Theorem 4.1 in \cite{anderson2010product}). 

\begin{theorem}[Anderson et al. (2010)] \label{thm:anderson} Let $(\SS,\CC,\RR)$ be a chemical reaction network endowed with mass action kinetics. Suppose the deterministic system is complex balanced with a complex-balanced equilibrium at $\vx^* \in \R_{>0}^s$. For a communicating class of states containing the initial state $\va_0$, the stationary distribution $\pi$ in the corresponding stochastic model is given by
\ba
\pi(\va) \propto \frac{(\vx^*)^\va}{\va!}~, \quad \quad \va \in \Lambda_{\va_0}.
\ea
\end{theorem}
When the network satisfies RNDB, the stationary distribution can be written directly in terms of the stochastic model, and does not necessitate finding a steady state solution of the corresponding deterministic system. 

\begin{theorem} \label{thm:RNDBstatdist}
Let $(\SS,\CC,\RR)$ be a reversible chemical reaction network endowed with mass action kinetics and a fixed choice of rate parameters for which the system possesses reaction network detailed balance. For $1 \le i \le v$, let $k(\vu_i) = \frac{k_1(\vu_i)}{k_1(-\vu_i)}$ and let $\vk = (k(\vu_1), k(\vu_2), \ldots, k(\vu_v))$. For a communicating class of states containing the initial state $\va_0$, let $\rm{comp}_\vu(\va; \va_0) = \{\alpha_i \in \Z | 1 \le i \le v \}$ be such that $\va - \va_0 = \sum_{i=1}^v \alpha_i \vu_i$. Then the stationary distribution is given by
\ba \label{eq:RNDBstatdist}
\pi(\va) \propto \frac{1}{\va!} \vk^{\rm{comp}_\vu(\va; \va_0)} = \frac{1}{\va!} \prod_{i=1}^v k(\vu_i)^{\alpha_i}
\ea
\end{theorem}
\begin{proof} Since the system is reaction network detailed balanced, it is complex balanced and so by Theorem \ref{thm:anderson} if $\vx^*$ is an equilibrium of the deterministic system, then for $\va \in \Lambda_{\va_0}$
\ba
\pi(\va) \propto \frac{(\vx^*)^\va}{\va!} = \frac{(\vx^*)^{\va_0 + \sum_{i=1}^v \alpha_i \vu_i}}{\va!} \propto \frac{(\vx^*)^{\sum_{i=1}^v \alpha_i \vu_i}}{\va!}. 
\ea
Reaction network detailed balance implies that for $\vu \in V(\RR)$, $(\vx^*)^\vu = k(\vu)$, and so 
\ba
\pi(\va) \propto  \frac{1}{\va!} \prod_{i=1}^v ((\vx^*)^{\vu_i})^{\alpha_i} = \frac{1}{\va!} \prod_{i=1}^v k(\vu_i)^{\alpha_i}. 
\ea
Uniqueness of $\pi$ implies that the solution is independent of the choice of the sequence $(\alpha_i | 1 \le i \le s)$, and of the reference state $\va_0$. 
\end{proof}

\paragraph{Examples} \label{sec:examples}
\begin{enumerate}

\item Consider the network presented in figure \ref{fig:horn_ex}, which was studied by Horn and Jackson \cite{horn1972general}, and by Feinberg \cite{feinberg1989necessary}. 
\begin{figure}[h!]
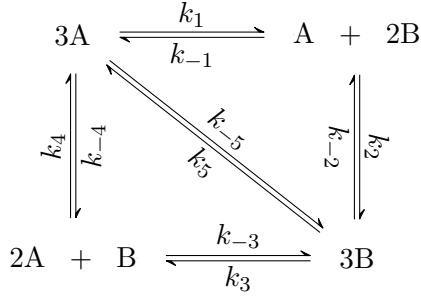

\centering
 \schemestart
 3A
 \arrow(1--2){<=>[$k_{1}$][$k_{-1}$]}
A \+ 2B  
 \arrow(2--3){<=>[$k_2$][$k_{-2}$]}[-90]  
3B
\arrow(3--4){<=>[$k_{-3}$][$k_{3}$]}[180] 
2A \+ B
\arrow(4--1){<=>[$k_4$][$k_{-4}$]}[90]  
 \arrow(@1--@3){<=>[$k_{-5}$][$k_{5}$]}[-45]
\schemestop
\caption{Network studied by Horn and Jackson in \cite{horn1972general} and Feinberg in \cite{feinberg1989necessary}} \label{fig:horn_ex}
\end{figure}
Reaction network detailed balance requires that for some $k >0$, the following conditions should hold on the rate constants:
\begin{align*}
\left(\frac{k_1}{k_{-1}}, \frac{k_2}{k_{-2}}, \frac{k_3}{k_{-3}}, \frac{k_4}{k_{-4}}, \frac{k_5}{k_{-5}}\right)  = \left(k^2, k, \frac{1}{k^2} , \frac{1}{k} , \frac{1}{k^3} \right)
\end{align*}
It is easy to see that if $\va(t) = (a(t),b(t))$ represents the state of the system at time $t$, $a(t)+b(t)$ is independent of $t$. Let $(0,a_0)$ be a reference state in the communicating class containing $(a,a_0 -a)$ for $0 \le a \le a_0$. The reaction vectors for this system are $\vu_1=(-1,1)$ and $\vu_2=(1,-1)$. We wish to find $(\alpha_1,\alpha_2)$ such that $\alpha_1 \vu_1 + \alpha_2 \vu_2 = (a,a_0-a) - (0,a_0) = (a,-a) = a (1,-1)$. One possible choice is $(\alpha_1,\alpha_2) = (0,a)$. So by Theorem \ref{thm:RNDBstatdist},  
\ba
\pi((a,a_0-a)) &\propto \frac{1}{a! (a_0 - a)!} \left(\frac{1}{k}\right)^a
\mbox{, or upon normalization } \nonumber\\
\pi((a,a_0-a)) &= \left(\frac{k}{1+k}\right)^{a_0} \left(\frac{1}{k}\right)^a {a_0 \choose a}
\ea
\item The following is a well-known model of the phosphofructokinase reaction as part of glycolysis cycle and is adapted from Gatermann et al. \cite{gatermann2005toric}. 
\ba
B \stackrel[k_{2}]{k_{-2}}{\rlas}  0 \stackrel[k_{-1}]{k_1}{\rlas} A \stackrel[k_{-3}]{k_3}{\rlas} C \quad, \quad 2A + B  \stackrel[k_{-4}]{k_4}{\rlas} 3A  
\ea
Since $r(\vu) =1$ for all $\vu \in V(\RR)$, by Theorem \ref{thm:oneveconerxn} this network has RNDB if and only if it has MCDB. Let $\vu_1 := (1,0,0)$, $\vu_2 := (0,1,0)$, $\vu_3 := (-1,0,1)$, and $\vu_4 := (1,-1,0)$, then it is easy to see that the only nontrivial reaction cycle is $(-\vu_1, \vu_2, \vu_4)$. So RNDB requires that $k_{-1} k_2 k_4 = k_{1} k_{-2} k_{-4}$. $\Z_{\ge 0}^3$ forms a single communicating class. Applying Theorem \ref{thm:RNDBstatdist}, we find that 
\ba
\pi((a,b,c)) = e^{-n_1} e^{-n_2} e^{-n_3}  \frac{n_1^a}{a!} \frac{n_2^b}{b!} \frac{n_3^c}{c!}
\ea
where $n_1 = \frac{k_1}{k_{-1}}$, $n_2 = \frac{k_2}{k_{-2}}$, and $n_3 = \frac{k_1}{k_{-1}}\frac{k_3}{k_{-3}}$. 
\item As a final example, we calculate the stationary distribution of the Markov chain arising from the CRN in figure \ref{fig:crn_ex}. This reversible CRN has 5 species, 6 complexes and 5 pairs of reversible reactions. Let ${\vect x} = (a,b,c,d,e) \in \Z_{\ge 0}^5$ be the state of the system. The stoichiometric subspace is the linear span of the following set of reaction vectors: $\{(-1,1,0,0,0),(0,1,0,-1,-1),(1,-1,1,0,0) \}$. Elementary linear algebra provides two conserved quantities, $a+b+d = l_1$ and  $a+b+e = l_2$. So a positive stoichiometric compatibility class is specified by two positive numbers $l_1$ and $l_2$. We write ${\vect x} = (a,b,c,l_1 -a -b ,l_2 - a -b)$ and take the reference state to be ${\vect x}_0 := (0,0,0,l_1 ,l_2 )$ and so 
\[
{\vect x} - {\vect x}_0 = (a,b,c,-a-b,-a-b) = (c-a)(-1,1,0,0,0) + (a+b)(0,1,0,-1,-1) + c (1,-1,1,0,0).
\]
 Thus we find the stationary distribution to be:
\ba
\pi({\vect x}) \propto \frac{1}{a! b! c! (l_1-a-b)! (l_2-a-b)!} \left(\frac{k_4 k_{-1}}{k_{-4} k_1}\right)^{a} \left(\frac{k_4}{k_{-4}}\right)^{b} \left(\frac{k_1 k_5}{k_{-1}k_{-5}}\right)^{c}
\ea
\end{enumerate}

\subsection*{Appendix: Birth and death processes} \label{sec:b&d}
This section is a slight digression, we study an example of a birth and death process arising from a chemical reaction network. Generically such a system has MCDB but not RNDB. If the deterministic system has multiple stable steady states, and if the stochastic system has a stationary distribution then we may expect it to be a bimodal distribution.  

When $V(\RR) = \{\vu, -\vu\}$ and $r(\vu) \ge 2$ the network, in general, has MCDB (as shown in Theorem \ref{thm:b&d}) but not RNDB. Satisfying MCDB does not require any constraints on the rate parameters, however RNDB involves at least one condition, and therefore at least one less degree of freedom in the available parameter space.  Network 2 in Section \ref{sec:motexamples} provides one example. Another is given by the following network:
\ba \label{eq:onerxnatom}
0 \stackrel[k_{-1}]{k_1}{\rlas} A \quad, \quad 2A \stackrel[k_{-2}]{k_2}{\rlas} 3A
\ea
The network in \eqref{eq:onerxnatom} is notable because it is (in a sense described below) the simplest example of a fully open network with multiple positive mass-action stable steady states. See for instance \cite{joshi2013complete,joshi2012atoms}, where a close relative of \eqref{eq:onerxnatom} -- obtained by making the reaction $2A \to 3A$ irreversible -- is studied and shown to be the ``{\em smallest atom of multistationarity}". A fully open network is a chemical reaction network where every chemical species is in inflow and outflow, thus for every species $A$ in the network a reaction of the type $0 \rlas A$ is included in the network. An atom of multistationarity is a fully open network which has the property of possessing multiple positive steady states (for some positive parameter values) and is a minimal network with respect to this property in the sense that if a single species or a single reaction is removed then the network loses multistationarity. The network \eqref{eq:onerxnatom} is a minimal element within the class of fully open networks which admit multiple positive steady states that are {\em stable} (see \cite{joshi2014survey}). We extend the notion of {\em atom of multistationarity} to that of {\em atom of multistability}, to be a network which admits multiple stable steady states and is minimal with respect to possessing this property.

It is a fairly simple exercise to find reaction rate constants for which \eqref{eq:onerxnatom} does not possess RNDB and has multiple positive stable steady states. For instance, consider the rate constants below:
\ba \label{eq:onerxnatom_example}
0 \stackrel[3.335]{13.5}{\rlas} A \quad, \quad 2A \stackrel[0.001]{0.132}{\rlas} 3A
\ea
The network above possesses two stable steady states at $a_1 = 5$ and $a_3 = 100$ and an unstable steady state at $a_2 = 27$. The stationary distribution is easily calculated to be 
\ba \label{eq:statdistonerxn}
\pi(a) = \frac{\Gamma}{a!} \prod_{i=0}^{a-1} \frac{k_{1}+k_{2}i(i-1)}{k_{-1}+k_{-2}i(i-1)} \quad, \quad a \in \Z_{\ge 0}
\ea
where $\Gamma$ is the normalization constant. Figure \ref{fig:statdistonerxn} depicts a graph of the equation in \eqref{eq:statdistonerxn}. The stationary distribution is seen to be bimodal with the two modes located approximately at the stable steady states of the deterministic system. 
\begin{figure}[H] 
\center{\includegraphics[scale=0.6]{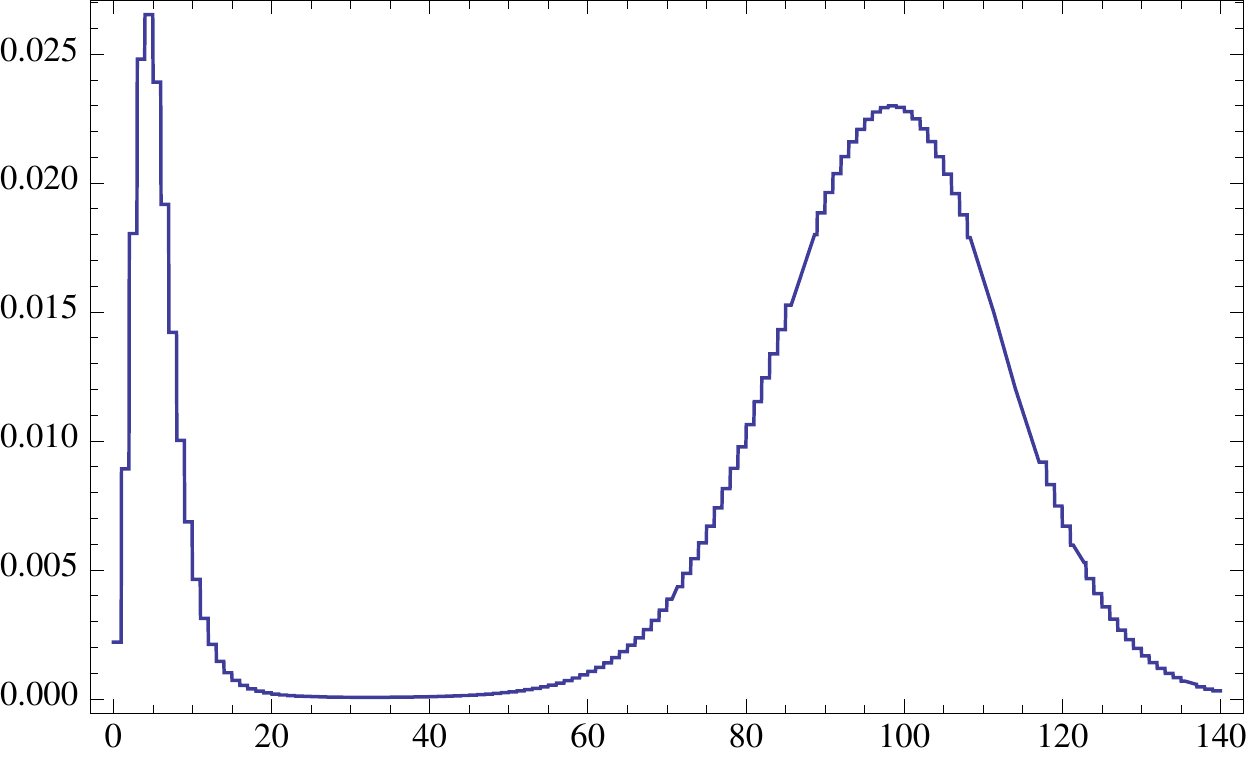}}
\caption{Plot of the stationary distribution for the network in \eqref{eq:onerxnatom_example}}
\label{fig:statdistonerxn}
\end{figure}

\subsection*{Acknowledgments}
I am extremely grateful to David Anderson for insightful comments and discussions, and pointing me to the relevant references. I would also like to thank the organizers of the American Institute of Mathematics (AIM) workshop on {\it Mathematical problems arising from biochemical reaction networks} -- Alicia Dickenstein, Jeremy Gunawardena, and Anne Shiu. The project of this paper originated from the stimulating discussions at the AIM workshop. I would like to thank the two anonymous referees for a careful reading and valuable suggestions for improvement of the manuscript. 

\bibliographystyle{amsplain}
\bibliography{detbal}

\end{document}